\newcommand{\MA}{\mathfrak{A}}

\newcommand{\ZZ}{\mathbb{Z}}
\newcommand{\CW}{\mathcal{W}}

\newcommand{\LLa}{\mbox{LlogL}}
\newcommand{\id}{I}

\newcommand{\Div}{\mbox{div}}
\newcommand{\DeltaA}{\Delta}
\newcounter{lil11}

\newcounter{lil22}

\newcommand{\baray}{\begin{array}{rcl}}
\newcommand{\earay}{\end{array}}
\newcommand{\barray}{\begin{array}{rcl}}
\newcommand{\earray}{\end{array}}

\newcommand\dela[1]{}

\newcommand{\bcase}{\begin{cases}}
\newcommand{\ecase}{\end{cases}}

\newcommand\del[1]{}

\del{

\newcommand\del[1]{}
}

\def\eps{\varepsilon}
\newcommand{\bu}{\bar u}
\newcommand{\bv}{\bar v}
\newcommand{\lk}{\left}
\newcommand{\lqq}{\lefteqn}
\newcommand{\rk}{\right}
\newcommand{\la}{{\langle}}
\newcommand{\ra}{{\rangle}}

\newcommand{\LL}{{\rm I \kern -0.2em L}}

\newcommand{\ep} {\varepsilon }

\newcommand{\be} {\begin{enumerate} }
\newcommand{\ee} {\end{enumerate} }

\newcommand{\CO}{{{ \mathcal O }}}

\newcommand{\CH}{{{ \mathcal H }}}

\newcommand{\BF}{{{ \mathbb{F} }}}
\newcommand{\CF}{{{ \mathcal F }}}

\newcommand{\RR}{{\mathbb{R}}}

\newcommand{\NN}{\mathbb{N}} %\newcommand{\OWi}{\mbox{ otherwise }}

\newcommand{\di}{\mbox{div}}

\newcommand{\PP}{{\mathbb{P}}}

\newcommand{\EE}{ \mathbb{E} }

\newcommand{\DEQS}{\begin{eqnarray*}}
\newcommand{\EEQS}{\end{eqnarray*}}
\newcommand{\DEQSZ}{\begin{eqnarray}}
\newcommand{\EEQSZ}{\end{eqnarray}}
\newcommand{\DEQ}{\begin{eqnarray}}
\newcommand{\EEQ}{\end{eqnarray}}

\newcounter{lil1q}
\newenvironment{steps}
{\begin{list} { \bf Step (\Roman{lil1q})}
{ \usecounter{lil1q}
\setlength{\leftmargin}{0.0cm}
\setlength{\topsep}{0.2cm}
\setlength{\itemsep}{0.0cm}
\setlength{\parsep}{0.1cm}
\setlength{\itemindent}{0.8cm}
\setlength{\parskip}{0.0cm}}}
{\end{list}}
\documentclass[11pt,english]{amsart}%article}
\usepackage[normalem]{ulem}

\usepackage{color}
\usepackage{amsmath}
\usepackage{relsize}
\usepackage{amsfonts,amssymb, color}
\usepackage{latexsym}
\input colordvi

\usepackage{enumitem}

\usepackage{amsfonts}

\usepackage{enumitem}

\theoremstyle{plain}

\newtheorem{theorem}{Theorem}[section]
\newtheorem{notation}{Notation}[section]
\newtheorem{claim}{Claim}[section]
\newtheorem{lemma}[theorem]{Lemma}%[section]
\newtheorem{corollary}[theorem]{Corollary}%[section]
\newtheorem{hypo}[theorem]{Assumption}%[section]
\newtheorem{definition}[theorem]{Definition}%[section]
\newtheorem{remark}[theorem]{Remark}%[section]

\newtheorem{proposition}[theorem]{Proposition}%[section]
\numberwithin{equation}{section}

\numberwithin{equation}{section} \allowdisplaybreaks

\begin{document}

\title[Uniqueness of 1d stochastic Keller--Segel model]{
Uniqueness of the stochastic Keller--Segel model in one dimension}

\author{Erika Hausenblas}
   \address{%
   Department of Mathematics,
	Montanuniversitaet Leoben,
	Austria.}
\email{erika.hausenblas@unileoben.ac.at}

\author[Debopriya Mukherjee]{Debopriya Mukherjee}

\address{%
   Department of Mathematics,
	Montanuniversitaet Leoben,
	Austria.}
\email{debopriya.mukherjee@unileoben.ac.at}

\author[Thanh Tran]{Thanh Tran}
\address{%
School of Mathematics and Statistics,
The University of New South Wales,
Sydney, Australia.}
 %\email{thanh.tran@unsw.edu.au}‎

%\subjclass
\date{\today}
\thanks{The first author of the paper is supported by Austrian Science
	Foundation, project number P 32295. The second author is supported by
	Marie Sk{\l}odowska-Curie Individual Fellowships H2020-MSCA-IF-2020,
	888255. The third author is partially supported by Australian Research
	Council, Discovery Project grant DP200101866.
}

\begin{abstract}
	In a recent paper (J. Differential Equations, \textbf{310}: 506--554, 2022), the authors proved the existence of martingale	solutions to a stochastic version of the classical Patlak--Keller--Segel
	system in 1 dimension (1D), driven by time-homogeneous spatial Wiener
processes. The current paper is a continuation and consists of two
results about the stochastic Patlak--Keller--Segel system in 1D. First, we
establish some additional regularity results of the solutions. The additional regularity is, e.g. important for its numerical modelling. %Secondly, this paper. 
Then, as a second result, we obtain the pathwise uniqueness of the solutions to the stochastic Patlak--Keller--Segel system in 1D. Finally, we conclude the paper with the existence of strong solution to this system in 1D.
\end{abstract}

\maketitle

\textbf{Keywords and phrases:} {Chemotaxis, Keller--Segel model, Stochastic Partial Differential Equations, Stochastic Analysis, Mathematical Biology, pathwise uniqueness.  }

\textbf{AMS subject classification (2010):} {Primary 60H15, 92C17,  35A01;
Secondary 35B65, 35K87, 35K51, 35Q92.}

\section{Introduction}
This paper is a continuation of \cite{EH+DM+TT_2019}, where the authors proved
the existence of martingale solutions to a stochastic version of the classical Patlak--Keller--Segel system driven by time-homogeneous spatial Wiener processes.
 The existence of solutions is shown in the weak probabilistic sense. We start with obtaining some uniform bounds on the higher moments of the solution. Then, we aim to establish the existence of a unique, strong solution of the stochastic version of the classical Patlak--Keller--Segel system driven by time-homogeneous spatial Wiener processes in one dimension. 

%In this paper, we provide adequate regularity results such that these solutions are strong and unique.

In the proof of the existence of a solution, compactness arguments are employed. Doing so, the underlying probability space {gets} lost, and one needs to introduce a concept of probabilistic
weak solutions, i.e.\ martingale solutions (see Definition 2.2 in \cite{EH+DM+TT_2019}).
This paper aims to give sufficient conditions to ensure pathwise uniqueness of
the solutions to the system \eqref{chemonoisestrat}. By the Yamada--Watanabe
Theorem, it follows that a global solution exists on every stochastic  basis and is unique.
We proceed first by proving pathwise uniqueness of the solutions, and then we
apply a result of Yamada--Watanabe--Kurtz to show that these solutions are strong and unique in law.
\par
The Yamada--Watanabe theory has been well developed for stochastic equations under the influence of Wiener noise, see e.g.\ \cite{cherny,engelbert,jacod,martin1,tappe,qiao,yamada}, even so, only a few results are dealing with the variational setting, see \cite{roeckner_uniq}.

\noindent
\textbf{Problem description:}
Let $\mathfrak{A}=(\Omega, \CF,\BF,\PP)$ be a complete probability space equipped with a filtration $\BF=(\CF_t)_{t\ge 0}$  satisfying the usual
condition.
%, i.e.\
%\begin{itemize}
%\item[(i)]
% $\mathbb{P}$ is complete on $(\Omega, \CF)$,
% \item[(ii)]
%for each $t\geq 0$, $\CF_t$ contains all $(\CF,\mathbb{P})$-null sets,
%\item[(iii)]
%  and the filtration $\CF_t$ is right-continuous.
%\end{itemize}
Let $\CH_1$ and $\CH_2$ be some Bessel potential spaces to be specified
later.
Let $\mathcal W_1$ and $\mathcal W_2$ be two cylindrical Wiener processes defined over
$\mathfrak{A}$ on Hilbert spaces $\CH_1$ and $\CH_2$, respectively.
 Let us define the Laplacian $\Delta:=\dfrac{\partial^2}{\partial x^2}$ with the Neumann boundary conditions given by
\begin{equation}
\label{eqn:4.3} \left\{
\begin{array}{ll}
D(A) &:= \{ u \in H^2(0,1):u_x(0)=u_x(1)=0 \},\cr
A u&:=\Delta u= u_{xx}, \quad u \in D(A).
\end{array}
\right.
\end{equation}
In this paper, we consider the following equation
\begin{align}\label{chemonoisestrat}
\lk\{ \barray
& d {u} - \lk( r_u\Delta   u- \chi \Div( u\nabla v) \rk)\, dt  =  u\circ  d \mathcal W_1, %\quad x\in I,\, t> 0,
\\
 &d{v} -(r_v \Delta v  -\alpha v)\, dt = \beta u \, dt + v \circ d \mathcal{W}_2,% \quad x\in I, t> 0,
\earray\rk.
\end{align}
where  $\CO=[0,1]$, $\beta\ge0$ is the production weight corresponding to $u$.
The positive terms $r_u$ and $r_v$  are the diffusivity of the cells and
chemoattractant, respectively. 
%\comca{(Priya: $\chi$ and $\alpha$ are not
%defined. Same for $\gamma$ in (1.3). Later on in (3.3) you used the positivity
%of these constants, so make sure that you mention it here.)}
The positive value $\chi$ is the chemotactic sensitivity, $\alpha\ge0$ is the so-called damping constant, $\beta\ge0$ is the production weight corresponding to $u$. The initial conditions are given by $u(0)=u_0$ %\in L^ {2}(I)$
and $v(0)=v_0$.
Since we model an intrinsic noise,  the stochastic integral above is interpreted in
the Stratonovich sense denoted by $\circ$. 
\del{To introduce the underlying Wiener processes, for $i=1,2$, let  $\{\psi^i_k:k\in\mathbb{N}\}$ be the eigenfunctions of $-\Delta$ with Neumann boundary conditions and $\{\beta^i_k:k\in\mathbb{N}\}$ be a family of mutually independent identically distributed standard Brownian motions.~The Wiener processes $\mathcal W_i$; $i=1,2$, are
given by
\begin{align}\label{def Wi}
\mathcal W_i(t,x)=\sum_{k\in\mathbb{N}} \lambda_k \beta^{(i)}_k(t) \psi^i_k(x), \quad x \in \CO,\,\, t >0,
\end{align}
where  $\{\lambda_k:k\in\mathbb{N}\}\in l^2(\NN)$,  is a non--negative sequence. In addition, let $\gamma=\sum_{k\in\NN}\,\lambda_k^2$.
%The Wiener process is a $L^2(\CO)$--valued stochastic process with  covariance operator  $Q=\CS\CS^T$,
%where $\CS$ is given by
%\DEQSZ\label{def_ss}
%\CS w &:=&\sum_{k\in\mathbb{N}} \lambda_k\la \psi_k,w\ra\, \psi_k, \quad w\in L^2(\CO).
%\EEQSZ
 %In particular, $\sum_{k\ge 1}\lambda_k^2<\infty$.
%with $\gamma>\frac 12$.
%
}
\par
An important limitation of the Stratonovitch stochastic integral is that the
Stratonovitch integral is not a martingale and the Burkholder--Davis--Gundy
inequality does not hold here. Consequently, it is more appropriate to work the
equation in the It\^o form. Following the technical analysis involved in the
conversion between the It\^o and Stratonovich form, one can see that
the system \eqref{chemonoisestrat} is equivalent to the following system:
 %For our notational convenience, we choose $\alpha=\beta=1$.
%In addition, the integral above is interpreted as Stratonovich sense.
%Hence, it is convenient to work the equation in the It\^o form, which is presented below: 
\DEQSZ\label{chemonoise}
\lk\{ \barray
 d {u} & -& \lk( r_u\Delta   u- \chi \Div( u\nabla v)\rk)\, dt  = u  d\mathcal W_1+ \gamma u \,dt, %\quad x\in I,\, t> 0,
\\
d{v}& -&(r_v \Delta v  -\alpha v)\, dt = \beta u \, dt + v d\mathcal{W}_2,%+ S_\gamma^2v\, dt .\phantom{\big|}% \quad x\in I, t> 0.
\earray\rk.
\EEQSZ
where $\gamma=\gamma_1 \leq \sum_{k=1}^{\infty}(\lambda_k^1)^2$ (see Remark \ref{rem wiener}).
For further details regarding the form of the correction
term and the technical analysis involved in the conversion, we refer to
\cite{WongZakai}, \cite[p. 65, Section 4.5.1]{Duan+Wang}.

%\section{The main result}
%
%
%
For completeness we add the definition of strong solutions to system \eqref{chemonoisestrat}. For a Banach space~$E$, the space $C_b^{0}([0,T];E)$ is the set of all continuous and bounded functions $u:[0,T]\to E$.
\begin{definition}
We call a pair of processes $(u,v)$ a strong solution to system \eqref{chemonoisestrat} if $u:[0,T] \times \Omega \rightarrow L^2(\CO)$ and $v:[0,T] \times \Omega \rightarrow
H^1_2(\CO)$
 are $\BF$--progressively  measurable processes\footnote{The process
 $\xi:[0,T]\times \Omega\to \mathbb{X}$ is said to be progressively measurable
over a probability space $(\Omega,\CF,(\CF_t)_{t\in[0,T]},\PP)$ if, for every
time $t\in[0,T]$, the map $ [0,t]\times \Omega \to \mathbb {X} $ defined by
$(s,\omega )\mapsto \xi_{s}(\omega )$ is $\mathcal {B} ([0,t])\otimes
{\mathcal {F}}_{t}$-measurable. This implies that $\xi$ is
$(\CF_t)_{t\in[0,T]}$-adapted.} such that $\PP$-a.s.\ $(u,v)\in C_b^0([0,T];L^2(\CO)) \times C_b^0([0,T];H^1_2(\CO))$ and satisfy for all $t\in [0,T]$ and $\PP$--a.s. the integral equation
\begin{align*} %\label{eq.strong}
u(t)
& =e^{t \,r_u\DeltaA} u_0
-\chi\int_0^ t e^{(t-s)r_u\DeltaA}  \Div( u (s)\,\nabla v(s)) \, ds
% \notag\\
%
%&\quad
+\int_0^ t e^{(t-s)r_u\DeltaA} u(s) \circ d\mathcal W_1(s),\\%\end{align} \begin{align}
v(t)
& =e^{t(r_v\DeltaA-\alpha I)} v_0
% -\alpha \int_0^ t e^{(t-s)(r_v\Delta-\alpha I)} v(s)\, ds
+\beta \int_0^ t e^{(t-s)(r_v\DeltaA-\alpha I)} u(s)\, ds + %\\&\quad+
\int_0^ t e^{(t-s)(r_v \DeltaA-\alpha I)} v(s)\circ  d\mathcal W_2(s).
\end{align*}
\end{definition}
%
%\blue{
%For completness, we will define also the  martingale solution, which is also known as a weak solution in the probabilistic sense.
% we define here also.
%
%

Let us consider the complete orthonormal system of the underlying Lebesgue space~$L^2(\CO)$  given by the trigonometric functions (see \cite[p.\ 352]{Birkhoff+Rota}) %$x\in\CO$
\begin{equation}\label{ONS}
\psi_k(x)=
\begin{cases}
{\sqrt{2}} \, \sin\big(2\pi{k} x\big) &\!\!\text{if } k\geq 1,\,x\in\CO, \\
{1} &\!\!\text{if } k = 0,\, x\in\CO, \\
{\sqrt{2}}\, \cos\big(2 \pi{k} x\big) &\!\!\text{if } k \leq - 1, x\in\CO.
\end{cases}
\end{equation}
For the proof of the existence of the solution, the Wiener
perturbation needs to satisfy regularity assumptions given in the next hypotheses.
\begin{hypo}\label{wiener}
Let $\CH_1$ and $\CH_2$ be two Bessel potential spaces such that
the embeddings  $\iota_1:\CH_1 \hookrightarrow L^2(\CO)$ and $\iota_2:\CH_2\hookrightarrow
H^1_2(\CO)$ are  Hilbert--Schmidt operators. The Wiener processes $\mathcal W_1$ and $\mathcal W_2$ are two cylindrical processes on  $\CH_1$ and $\CH_2$.
 \end{hypo}

\begin{remark}\label{rem wiener}
%As an example we can take $\CH_1=\CH_2=H^{(\delta)}$, where $\delta>2$.
%\red{(What is this space?)}
%Here, \red{(Where? I don't see it above.)} $\{\psi^{(\delta)}_k:k\in\mathbb{Z}\}$ with $\psi_k^{(\delta)}:=(1+({2\pi k})^2)^{-\delta/2}\psi_k$, where $\psi_k$ is given by \eqref{ONS}, forms an  orthonormal system in $H^{(\delta)}$.
%Let $\{\beta^j_k:k\in\mathbb{Z}\}, j=1,2,$ be two families of mutually independent identical distributed standard Brownian motions.
%The Wiener processes $\mathcal W_1$ and $\mathcal W_2$  are then given by %time homogenous spatial Wiener processes
%such that
%$$
%\mathcal W_j(t,x)=\sum_{k\in\mathbb{Z}} %\lambda_k
%\psi^{(\delta)} _k(x)\beta^j_k(t),\quad j=1,2.
%$$
%where  $\delta>1$.
%$\{\lambda_k:k\in\mathbb{Z}\}$, is a non--negative sequence with $\lambda_k=\lambda_{-k}$, for all $k\in\ZZ$.
%In addition, let $\gamma=\sum_{k\in\NN}\,\lambda_k^2$. %In particular, $\sum_{k\ge 1}\lambda_k^2<\infty$.
%with $\gamma>\frac 12$.
%
As an example we can take $\CH_1=H^{\delta_1}(\CO)$ for $\delta_1>1$, and
$\CH_2=H^{\delta_2}(\CO)$ for $\delta_2>2$ where $H^{\delta_i}(\CO)$, $i=1,2$, is to be
defined in Notation~\ref{Hdelta}. The Wiener processes $\mathcal W_1$ and $\mathcal W_2$  are
then given by
\[
\mathcal W_1(t,x)=\sum_{k\in\mathbb{Z}}
\psi^{(\delta_1)} _k(x)\beta^{(1)}_k(t) \quad\text{and}\quad
\mathcal W_2(t,x)=\sum_{k\in\mathbb{Z}}
\psi^{(\delta_2)} _k(x)\beta^{(2)}_k(t)
\]
where $\psi_k^{(\delta_i)}:=(1+({2\pi k})^2)^{-\delta_i/2}\psi_k$, $i=1,2$, with $\psi_k$ defined by
\eqref{ONS}. Here~$\{\beta^{(i)}_k:k\in\mathbb{Z}\}$, $i=1,2,$ are two families of mutually
independent identically distributed standard Brownian motions. We note
that~$\{\psi^{(\delta_i)}_k:k\in\mathbb{Z}\}$ forms an  orthonormal system in
$H^{\delta_i}(\CO)$, $i=1,2$. We will also write
\begin{align}\label{def Wi}
\mathcal W_i(t,x)=\sum_{k\in\mathbb{Z}} \lambda^{(i)}_k
\psi _k(x)\beta^{(i)}_k(t),\quad i=1,2,
\end{align}
where, for~$i=1,2$, the sequence $\{\lambda^{(i)}_k:k\in\mathbb{Z}\}$, $\lambda^{(i)}_k:=(1+\mu_k)^{-\delta_i/2}=(1+{(2\pi k)^2})^{-\delta_i/2}$, is non-negative with $\lambda^{(i)}_k=\lambda^{(i)}_{-k}$ for
all $k\in\ZZ$,
where $\mu_k=(2\pi k)^2$ are the corresponding eigenvalues of $-\Delta$ with Neumann boundary conditions.
Since we will need it later on,  for each $i=1,2$, let us define the constant
% we define
\DEQSZ \label{defofgamma}
\gamma_i:= \sum_{k \in \NN} |\lambda_k^i|^2 |\psi_k|_{L^\infty}^2 \leq \sum_{k\in\ZZ}\,(\lambda^{(i)}_k)^2.
\EEQSZ 
%Since we will need it later on, let us define the constant  $\gamma:=\sum_{k\in\ZZ}\,\lambda_k^2$.
%In particular, $\sum_{k\ge 1}\lambda_k^2<\infty$.
%with $\gamma>\frac 12$.
\end{remark}

Since the solution $(u,v)$ represents the cell density and concentration of the chemical signal,
$u$ and $v$ are non-negative. This implies that the initial conditions $u_0$ and
$v_0$ are non-negative. Besides, one needs to impose some more regularity assumptions on $u_0$
and $v_0.$ % on the initial condition to obtain existence and  uniqueness of the solution.
\begin{hypo}\label{init}
Let $u_0\in L^ 2 (\CO)$ and $v_0\in H^ 1_2 (\CO)$ be two random variables over $\mathfrak{A}$ such that
\begin{enumerate}[label=(\alph*)]
  \item $u_0\ge 0$ and $v_0\ge 0$;
%In addition, there exists some $\ep>0$ such that the sets $\{\xi\in I:u_0(\xi)\ge \ep\}$ and $\{\xi\in I:v_0(\xi)\ge \ep\}$
%  have  non empty interior.
  \item $(u_0,v_0)$ is  $\CF_0$--measurable;
  \item  $\EE\left[|u_0|_{L^2}^2\right]<\infty$ and $\EE\left[|\nabla v_0|_{L^2}^2\right]<\infty$.
\end{enumerate}
\end{hypo}

\begin{notation}\label{Hdelta}
Let $A_1$ be the positive Laplace operator~$-\Delta$ (as an operator defined on~$L^2(\RR)$)
restricted to functions defined on~$\CO$, namely,
\[
A_1 :=-\Delta, \quad
D(A_1):= H^2_2(\CO)\cap H^1_0(\CO),
\]
and let $(\psi_j,\rho_j)$ be the eigenfunctions and eigenvalues of $A_1$.
The  Bessel potential space $H^\kappa(\CO)$ is defined by
\[
H^\kappa(\CO)=\Big\{u=\sum_{j}a_j \psi_j \in L^2(\CO):\|u\|_{H^{\kappa}(\CO)}
=\Big(\sum_{j}a_j^{2} \rho_j^{2\kappa}\Big)^{1/2}<\infty
\Big\}.
\]
\end{notation}
\begin{notation}
Let us  define by
$\LLa(\CO)$ the Zygmund space (see  \cite[Definition 6.1, p.\ 243]{bennet}) consisting of all Lebesgue-measurable functions $f:\CO\to \RR$ for which $\int_\CO|f(x)|\log(|f(x)|)^+\, dx<\infty$. This space is equipped with the norm
\[|f|_{\LLa}:= \int_0^1 f^{\ast}(t) \log(\dfrac{1}{t})dt,\]
where $f^{\ast}$ is defined by
\[ f^{\ast}(t)=\inf\{\lambda: \delta_f(\lambda)\leq t  \}, \quad t \geq 0.
\]
Here $\delta_f(\lambda)$ is the Lebesgue measure of the set $\{x \in \CO: |f(x)| > \lambda \}, \,\, \lambda \geq 0.$ We note that $f \in \LLa$ iff $\int_{\CO} |f(x)|\log (2+|f(x)|)\,dx < \infty$ (see p. 252 of \cite{bennet}). Note also that by Theorem 6.5 p.\ 247 in \cite{bennet} we have that $\LLa(\CO) \hookrightarrow L^1(\CO)$.
\end{notation}
Let us briefly describe the content of this paper.  In Section \ref{add reg}, we
formulate a proposition to provide some uniform bounds on $u$ and $v$ in such a
way that we can control the $L^{2}(\Omega;L^{\infty}(0,T;L^{2}(\CO))$
norm of $u$ and the $L^{2}(\Omega;L^\infty(0,T;H^{1}_2(\CO))$ norm of $v$.
%\comca{Priya: I'm not sure what you mean by $H^1_2$ and $H^2_2$ (previous page).
%Are they just typos?} 
We postpone the proof of this Proposition to Section \ref{add reg proof}.
In Section \ref{pathwise}, we deduce that if there exist two solutions on the
same probability space, then both the solutions are identical.
Here, we prove our main result on the existence of a unique strong solution to the system \eqref{chemonoisestrat} which we reformulated as
Corollary \ref{main_example}.  In the Appendix, we collect some elementary results which are needed in the course of analysis.
\section{Additional regularity}\label{add reg}
\newcommand{\bunn}{\bar{u}_n}
\newcommand{\bvnn}{\bar{v}_n}
In this section, we formulate the following proposition to obtain some
uniform bounds on $u$ and $v$ in such a way that we can control the
$L^{2}(\Omega;L^{\infty}(0,T;L^{2}(\CO))$ norm of $u$ and the
$L^{2}(\Omega;L^\infty(0,T;H^{1}_2(\CO))$ norm of $v$. 
%To prove the following proposition, we mark the corresponding steps.
%\begin{itemize}
%\item To start with, for $p>1$, we aim to bound the $p$-th moment of $L^1$ norm of $u$, i.e.,
%\begin{align*}
%\EE\Big[ \sup_{0 \leq s \leq T} |u(s)|^p_{L^1}\Big]
%&\leq C.
%\end{align*}
%\item In the next step, we obtain uniform bound for the $L^{2}(\Omega;L^{\infty}([0,T];L^{2}(\CO))$ norm of $u$ and the $L^{2}(\Omega;L^\infty([0,T];H^{1}_2(\CO))$ norm of $v$. 
While carrying out this formulation, in order to handle the non-linear
term $u \nabla v$, it is essential to obtain the bound for the $p$-th moment of
$H^1_2(\CO)$ norm of $v$ for $p>1$ i.e.,
\[
	v \in L^p(\Omega;L^p(0,T;H^1_2(\CO))).
\] 
%\comca{(Priya: I find your notations $H^1_2$ and $H^2_2$ strange. A Sobolev
%	space in $L^p$ is denoted by $W_p^k$ or $W^{k,p}$, and when $p=2$
%one writes $H^k$ because they are now Hilbert spaces, $H$ for Hilbert. This is
%due to Lions and all people in PDEs use these notations. The
%subscript $2$ is thus redundant in $H_2^k$. If you agree with me, then please
%correct all the remaining parts. If this is a standard notation in stochastic,
%then I will live with it, even though I find it strange.)}
To do so, we will apply Theorem 4.5 in \cite{vanNeerven1}. Consequently, we will end up working in the real interpolation space 
\begin{align*}
\big(H^1_2(\CO), H^{-1}_2(\CO)\big)_{1-\frac{1}{p},p}= {B^{1-\frac{2}{p}}_{2,p}}(\CO),
\end{align*}
which is a Besov space; see \cite[p.\ 1406]{vanNeerven1}.
\begin{proposition}\label{uniformboundtaun}
Let $T>0$ and $\mathfrak{A}=(\Omega,\CF,(\CF_t)_{t\in[0,T]},\PP)$ be a
probability space satisfying the usual conditions. Also, let $(\mathcal
	W_1, \mathcal W_2)$ be a pair of Wiener processes over $\MA$, and
$\CH_1$ and $\CH_2$ be two Bessel potential spaces which satisfy Assumption \ref{wiener}. Let
the initial data $(u_0,v_0)$ satisfy Assumption \ref{init} with
$\EE|u_0|_{L^1}^{12}< \infty$ and $\EE|v_0|^{12}_{B_{2,12}^{{5}/{6}}} < \infty$. If $(u,v)$ is a martingale solution  solution to the system \eqref{chemonoise}, then, there exist positive constants $a_0$ and $a_1$ such that we have the following inequality
\begin{align}
& \frac 14 \,\EE \Big[\sup_{0 \leq s \leq t  }|u(s)|^2_{L^{2}} \Big]
+ \frac 14 \,\EE \Big[\sup_{0 \leq s \leq t }|\nabla v(s)|^2_{L^{2}} \Big]
 + \frac{r_u}{4}\,\EE\Big[ \int_0^{t }|\nabla u(s)|^2_{L^{2}}\, ds \Big]
\notag\\
& \quad
 + \frac{r_v}{2}\, \EE\Big[ \int_0^ {t }\int_{\CO} |\Delta v(s,x)|^2\,dx\, ds\Big]
  + \alpha \EE\Big[ \int_0^ {t }\int_{\CO} |\nabla v(s,x)|^2\,dx\, ds\Big]
 % + r_v\, \EE\Big[ \int_0^ {t \wedge \tau_{n}(u)}\int_{\CO} |\nabla v(s,x)|^2 |D^2v(s,x)|^2\,dx\, ds\Big]
 \notag\\
&
 \leq a_0 \Big( \EE|u_0|_{L^2}^2
+ \EE|\nabla v_0|_{L^2}^2
 + T\, \EE|u_0|_{L^1}^8 e^{(\gamma^8 + \frac{1}{2}) T^8}
 + T \EE|v_0|^{12}_{B_{2,12}^{\frac{5}{6}}}
 \notag\\
& \quad + T^2 \EE |v_0|_{L^1}^{12} e^{(\gamma^{12} + \frac{1}{2}) T^{12}}\Big) e^{a_1 T}.
\end{align}
\end{proposition}
We postpone the proof of this proposition to Section \ref{add reg proof}.
 
%************************************************************************************************************************

%
%
\section{Pathwise uniqueness of the solution}\label{pathwise}
In this section, we prove pathwise uniqueness of the solution to the system \eqref{chemonoise}.
Thanks to %\sout{the theorem of T. Yamada and S. Watanabe} 
Yamada--Watanabe theorem, %\comc{\sout{which affirmsthat}} 
(weak) existence and pathwise uniqueness of the solution of a stochastic equation guarantee the existence of a strong solution.
Existence of a martingale solution to the given system is obtained in Theorem 2.6 of \cite{EH+DM+TT_2019}.
Accordingly, our aim is to obtain the pathwise uniqueness to show the existence of a unique, strong solution to the system \eqref{chemonoise} in one dimension.
In the following theorem, we will find under which conditions pathwise uniqueness holds.
For the sake of the completeness we  start with the definition of pathwise uniqueness.
\begin{definition} The equation \eqref{chemonoise} is pathwise unique if, whenever $(\Omega, \CF,\BF,\PP,(u_i,v_i),\mathcal W_1, \mathcal W_2)$, $i=1,2$, are solutions to \eqref{chemonoise} such that
$\PP(u_1(0)=u_2(0))=1$ and $\PP(v_1(0)=v_2(0))=1$, then
\[
\PP(u_1(t)=u_2(t))=1 \quad \mbox{and} \quad
\PP(v_1(t)=v_2(t))=1,
 \quad \mbox{for every } t \in [0,T].
\]
\end{definition}
\begin{theorem}\label{thm path uniq}
Let $T>0$ and $\mathfrak{A}=(\Omega,\CF,(\CF_t)_{t\in[0,T]},\PP)$ be a probability space satisfying the usual conditions. Also, let $(\mathcal W_1, \mathcal W_2)$ be a pair of Wiener processes over $\MA$, and
$\CH_1$ and $\CH_2$ be two Bessel potential spaces which satisfy Assumption \ref{wiener}.
Let the initial data  $(u_0,v_0) \in L^{2}(\CO) \times H^1_{2}(\CO)$
satisfy Assumption \ref{init}.
Let $(u_i,v_i)$; $i=1,2$, be two solutions to the system \eqref{chemonoise}
such that $\PP$--a.s. the following holds
\begin{align}\label{uL2 vL2}
\begin{cases}
u \in \mathbb{X}_u 
& \text{where } \mathbb{X}_u := 
L^{2}(\Omega;C([0,T];L^{2}(\CO))) 
\cap L^{2}(\Omega;L^{2}(0,T ;H^{1}_2( \CO))),
\\
v \in \mathbb{X}_v 
& \text{where } \mathbb{X}_v :=
L^{2}(\Omega;C([0,T];H^1_2(\CO))) \cap
L^{2}(\Omega;L^{2}(0,T;H^{2}_2(\CO))).
\end{cases}
\end{align}
 Then, the processes $(u_1,v_1)$ and $(u_2,v_2)$ are identical in $\mathbb{X}_u\times\mathbb{X}_v$.
\end{theorem}
%In case Theorem \ref{thm path uniq} holds,  a unique strong solution exists.
Since we are interested in the consequences of the pathwise uniqueness, in particular on  the existence of a unique strong solution,
 we will present the following Corollary and postpone the proof of
Theorem \ref{thm path uniq}.
\begin{corollary}\label{main_example}
Suppose the conditions in Theorem \ref{thm path uniq} are satisfied. Then, the system \eqref{chemonoise} admits a unique strong solution.
\end{corollary}

%\end{proof}

%
\begin{proof}
We show that there exists a unique strong solution to the system \eqref{chemonoise} in the framework of the variational approach. To do so, we use Theorem 1.7 of \cite{qiao}.
%\sout{Let $\mathfrak{A}=(\Omega,\CF,\mathbb{F},\PP)$ be the given filtered probability space.
%	Let $\mathcal W_1$ and $\mathcal W_2$ be two independent cylindrical
%	Wiener processes on $\mathcal H_1$ and $\mathcal H_2$ defined over
%$\mathfrak{A}$. }
	%\comca{(Already defined in Theorem~\ref{thm path uniq}.)} 
Define
\begin{gather*}
\CW:=(\CW_1,\CW_2),\quad \mathbb{U}:=\mathcal H_1\times \mathcal H_2,
\\
\mathbb{V}_1:=  H^{1}_2(\CO) \times H^{2}_2(\CO), \quad
\mathbb{H}:= L^{2}(\CO) \times H_{2}^{1}(\CO),
\quad
 \mathbb{V}_2:=   H^{-1}_2(\CO) \times L^{2}(\CO),
\end{gather*}
 where we use the Gelfand triple
 $H_{2}^1(\CO)  \hookrightarrow  L^{2}(\CO) \hookrightarrow  H_{2}^{-1}(\CO)$.
% \comca{Do you mean that $H_2^{-1}(\CO)$ is the dual of $H_2^1(\CO)$?}
 Let us define the following path space
 $$\mathbb{B}:=\lk\{ X=(u,v)^{\top}\in C([0,T];\mathbb{H}) : \int_0^ T|X(t)|_{\mathbb{V}_1}\, dt<\infty\rk\}.$$
% Therefore, we have for all $t \in [0,T]$
% \begin{align*}
% \int_0^t \Big[ |u(s)|_{L^{2}} + |v(s)|_{H^1_2} \Big]\,ds < \infty.
% \end{align*}
Furthermore, let us consider the maps
$$
b:[0,T]\times \mathbb{B}\to\mathbb{V}_2 \quad \hbox{and} \quad \sigma:[0,T]\times \mathbb{B}\to \mathcal{L}_2(\mathbb{U}, \mathbb{H})$$ defined by
$$
b(t,X):=\lk(\begin{matrix} r_u \Delta u(t)-\chi\,\Div( u(t) \nabla v(t)) +\gamma u(t)\\
r_v \Delta v (t) - \alpha v(t) + \beta u(t)\end{matrix}\rk),
$$
and
$$
\sigma(t,X)[h]:=\lk(\begin{matrix}   u(t) h_1 \\  v(t)h_2
\end{matrix}\rk) ,
$$ for $X= (u,v)^\top \in \mathbb{B}$, $h= (h_1,h_2)^\top\in\mathbb{B}$, and $t
\in [0,T]$.
%are well defined for $w=(u,v)^T$, $9u,v)$ being a solution to \eqref{sys1noise.ito}.
%\newcommand{\bu}{\bar u}
%\newcommand{\bv}{\bar v}

Let $\bar X= (\bu,\bv)^\top$ be a martingale solution of \eqref{chemonoise}. To show that $\bar X$ is a strong solution, we need to verify $\mathbb{P}-$a.s.
\begin{align}\label{esti bx sig}
\int_0^T |b(s,\bar X(s))|_{\mathbb{V}_2}\, ds
+ \int_0^T |\sigma(s, \bar X(s))|_{\mathcal{L}(\mathbb{U}, \mathbb{H})}\,ds < \infty.
\end{align}
We first estimate $|b(s,\bar X(s))|_{\mathbb{V}_2}$ as follows:
\begin{align}\label{esti bx}
&|b(s,\bar X(s))|_{\mathbb{V}_2}
\\
&= \Big| r_u \Delta \bu(s) - \chi \, \Div(\bu(s)\nabla \bv(s))-\gamma \bu(s) \Big|_{H_2^{-1}}
 + \Big|r_v \Delta \bv(s) - \alpha \bv(s) + \beta \bu(s) \Big|_{L^{2}}
\notag\\
& \leq  r_u |\Delta \bu(s)|_{H_2^{-1}}  
+ \chi | \Div(\bu(s)\nabla \bv(s))|_{H_2^{-1}} 
+ \gamma| \bu(s)|_{H_2^{-1}}   
 + r_v |\Delta \bv(s) - \alpha \bv(s)|_{L^{2}} + \beta |\bu(s)|_{L^{2}}
\notag\\
&\leq r_u | \bu(s)|_{H^{1}_{2}}
+ \chi | \bu(s)\nabla \bv(s)|_{L^{2}}
+ \gamma |(-\Delta)^{-1}\bu(s)|_{L^{2}}
 + (r_v + \alpha) | \bv(s)|_{H_2^{2}}  + \beta |\bu(s)|_{L^{2}}
 \notag\\
 & \leq (r_u+ c \gamma + c \beta) | \bu(s)|_{L^{{2}}}
+ \chi | \bu(s)\nabla \bv(s)|_{L^{2}}
 + (r_v + \alpha) | \bv(s)|_{H_2^{2}}.\notag
\end{align}
%Now, by Remark 5, page-176 Runst and Sickel for $s=1,\,\,p=q=2,\,\,$ we get
To estimate the second term on the right hand side, namely $| \bu(s)\nabla \bv(s)|_{L^{2}} $,
we use the embedding $H^1_2(\CO) \hookrightarrow L^\infty(\CO)$ and the
Cauchy--Schwarz inequality to obtain
\begin{align}\label{esti bu nab bv}
| \bu(s)\nabla \bv(s)|_{L^{2}}
\leq  | \bu(s)|_{L^{2}} |\nabla \bv(s)|_{L^{\infty}}
\leq  | \bu(s)|_{L^{2}} | \bv(s)|_{H^2_{2}}.
\end{align}
Using \eqref{esti bu nab bv} in \eqref{esti bx}, we have from Proposition \ref{uniformboundtaun},
\begin{align}\label{esti bx1}
\int_0^T |b(s,\bar X(s))|_{\mathbb{V}_2}\,ds
 &\leq (r_u+ c \gamma + c \beta) \int_0^T | \bu(s)|_{L^{{2}}}\,ds
+ \chi \int_0^T | \bu(s)|_{L^2} | \bv(s)|_{H^{2}_{2}}\,ds
\notag\\
& \quad
 + (r_v + \alpha) \int_0^T | \bv(s)|_{H_2^{2}}\,ds
 \notag\\
 & \leq (r_u+ c \gamma + c \beta)\int_0^T | \bu(s)|_{L^{{2}}}\,ds
+ \chi\Big\{\sup_{0 \leq s \leq T}| \bu(s)|_{L^2}  \int_0^T | \bv(s)|_{H^{2}_{2}} \,ds \Big\}
\notag\\
& \quad + (r_v + \alpha)\int_0^T | \bv(s)|_{H_2^{2}}\,ds < \infty.
\end{align}
% Also, $(\bu,\bv)$ is a martingale solution of \eqref{chemonoisestrat} and by Theorem 2.6 in \cite{EH+DM+TT_2019}  we have $\PP$-a.s.
% \begin{align}
% \int_0^T |\bu(s)|_{H^1_2}^2 \, ds + \int_0^T |\bu(s)|^{q+1}\,ds < \infty,
% \end{align}
% Using \eqref{esti bu nab bv} and \eqref{esti buq}, we conclude from \eqref{esti bx} that
% \begin{align}
% \int_0^ T|b(s,X(s))|_{\mathbb{V}_2}\, ds & \leq (r_u + \beta) \int_0^T |\bu(s)|_{L^{q+1}}
% + \chi \Big\{ \sup_{0 \leq s \leq T}|\bv(s)|_{L^2}
% \Big(\int_0^T |\bu(s)|_{H^1_2}\,ds \Big)
% \Big\}\notag
%\\& \quad
% + c \int_0^T |\bu(s)|^{q+1}_{L^{q+1}}\,ds
% + (r_v + \alpha) \int_0^T |\bv(s)|_{H^1_2}\,ds
% < \infty.
% \end{align}
 Also, using  Proposition \ref{uniformboundtaun}, we have $\PP$-a.s.
 \begin{align*}
 \int_0^T |\sigma(s,\bar X(s))|^2_{\mathcal{L}_2(\mathbb{U},\mathbb{H})}
\,ds
\leq \int_0^T \Big[
|\bu(s)|_{H^1_{2}}^2 + |\bv(s)|_{H_2^2}^2
\Big]\,ds < \infty   .
 \end{align*}
 This yields \eqref{esti bx sig}.
Finally, since $\bar X=(\bu,\bv)$ is a martingale solution,  the solution
process $\bar X$ with $x=(u_0,v_0)^\top$ $\PP$--a.s. satisfies
$$
\bar X(t)=x+\int_0^ t b(s,\bar X(s))\, ds + \int_0^ t \sigma(s,\bar X(s))d\CW(s),\quad t\in[0,T].
$$
In addition, by Theorem \ref{thm path uniq} we have pathwise uniqueness of the solution. Hence, by Theorem 1.7 of \cite{qiao}, we have shown that $\bar X(t)$ is the unique strong solution to the system \eqref{chemonoise}. This finishes the proof of the corollary.
\end{proof}

%\section{Maximal inequalities for the $v$ term to be summerized}
\begin{proof}[Proof of Theorem \ref{thm path uniq}]
\del{ In particular, there exists a constant $C=C(T)>0$ such that
\DEQSZ\label{finiteinlinfty}
\EE \sup_{0\le t\le T}|u_i(t)|^2_{H^1_2} &\le & C.
\EEQSZ}

Let us recall, since $(u_1,v_1)$ and $(u_2,v_2)$ are solutions to the system \eqref{chemonoise} with
$\PP(u_1(0)=u_2(0))=1$ and $\PP(v_1(0)=v_2(0))=1$, we can write
\DEQSZ
d{u}_i(t)&=& \Big( r_u \Delta u_i (t) %\mbox{div}\big(  \nabla u(t)\big)
-\chi \mbox{div} \big( u_i(t)\nabla v_i(t)\big) + \gamma u_i(t)\Big)\, dt
+ u_i(t)\,d\mathcal W_1(t)\label{sysu12.11}
\\
d{v}_i(t) &=& \big(r_v \Delta v_i(t)+\beta u_i(t) %K(u(t),v(t))
-\alpha v_i(t)\Big)\, dt + v_i(t)\, d \mathcal W_2(t),\,\quad t\in [0,T],\phantom{\big|}\,\,i=1,2.\label{sysv12.11}
\EEQSZ
%In addition, both solution processes are  \cadlag on $[0,\tau_m\wedge T]$.
%
%
%\medskip
%
In the first step we will introduce a family of stopping times $\{\tau_N:N\in\NN\}$,
and show that on the time interval $[0,\tau_N]$ the solutions $u_1$ and $u_2$, respective, $v_1$ and $v_2$, are indistinguishable.
In the second step, we will show that $\PP\lk( \tau_N<T\rk)\to 0$  for $N\to\infty$. From this follows that $u_1$ and $u_2$ are indistinguishable on the time interval $[0,T]$.
%\medskip
\paragraph{\bf Step I:}
Let us introduce the stopping times  $\{ \tau_N:N\in\NN\}$ as follows: let %and  $\{ \tau^2_N:N\in\NN\}$  given by
\DEQS
\tau_{N,i}^1 &:=&{\inf\{t \geq 0: \sup_{s \in [0,t)} |u_i(r)|_{L^{2}} \ge N \} \wedge T}; \quad i=1,2,
\\
\tau_{N,i}^2 &:=&{\inf\{t \geq 0: \sup_{r \in [0,t)} |u_i(r)|_{L^1} \ge N \} \wedge T}; \quad i=1,2,
%{\color{magenta}\tau_{N,i}^2}&:=&\inf\{t \geq 0: |1_{[0,t)} v_i|_{\BH_\rho} \geq N \} \wedge T, \quad i=1,2,
%\\
%{\color{magenta}\tau_{N,i}^3}&:=&\inf\{t \geq 0: \|u_i\|_{L^{q+1}(0,t;L^{q+1})} \geq N \} \wedge T\quad i=1,2,
\EEQS
%with $s_2=1-\gamma$ and for any $\gamma<1$,
and $\tau_N:=\min_{i=1,2}(\tau_{N,i}^1,\tau_{N,i}^2)$. %,\tau_{N,i}^4)$.
The aim is to show
that $(u_1,v_1)$ and $(u_2,v_2)$ are indistinguishable on the time interval $[0,\tau_N]$. %, {with $\tau_N=\tau_N$.} %\inf(\tau^1_N, \tau^2_N)$.}

\medskip
Fix  $N\in\NN$.
To get uniqueness on $[0,\tau_N]$ we first stop the original solution processes at time $\tau_N$ and extend the processes $(u_1,v_1)$ and $(u_2,v_2)$
by other processes to the whole interval $[0,T]$.  For this purpose,
let $(y_1,z_1)$ be  solution to
\DEQSZ\label{eq1}
 \lk\{\barray dy_1(t) &=&\Delta y_1(t)\, dt
 + y_1 (t)\,d\theta_{\tau_N}\circ \mathcal W_1(t),\quad t \geq 0,
\\
 dz_1(t) &=& \lk(\Delta z_1(t)-\alpha z_1(t)\rk)\,dt
 + z_1 (t)\,d\theta_{\tau_N}\circ \mathcal W_2(t),\quad t \geq 0,
\earray\rk.  \EEQSZ
with initial data $ y_1(0):=u_1(\tau_N)$,  $ z_1(0):=v_1(\tau_N)$, and
let $(y_2,z_2)$ be  solutions to
\DEQSZ\label{eq22}
 \lk\{\barray dy_2(t) &=& \Delta y_2(t)\, dt + y_2 (t)\,d\theta_{\tau_N}\circ \mathcal W_1(t),\quad t\geq 0 ,
\\
 dz_2(t) &=& \Delta z_2(t)-\alpha z_2(t)+ z_2 (t)\,d\theta_{\tau_N}\circ \mathcal W_2(t),\quad t\geq 0,
\earray\rk.
 \EEQSZ
with initial data $ y_2(0):=u_2(\tau_N)$ and  $ z_2(0):=v_2(\tau_N)$. Here, $\theta$ denotes the shift operator, i.e.\ $\theta_\tau\circ W_i(t)=W_i(t+\tau)$, $i=1,2$.
Since $(u_1,v_1)$ and $(u_2,v_2)$ are continuous in $L^{2}(\CO)\times L^2(\CO)$, $(u_1(\tau_N),v_1(\tau_N))$ and $(u_2(\tau_N),v_1(\tau_N))$ are well defined and
 belong $\PP$--a.s.\ to $L^{2}(\CO)\times L^2(\CO)$.
%By Theorem 2.5.1 in \cite{roeckner}, we know that there exists a unique solution to  \eqref{eq1} belonging $\PP$-a.s.\ to $C_b([0,T];L^{q+1}(\CO))$ and by Theorem 4.5. \cite{vanNeerven1} we know that there exists a unique solution to\eqref{eq22} belonging $\PP$-a.s.\ to
%$C_b([0,T];L^2(\CO))$.}
Now, we define two new couple of processes in such a way that they coincide with $(u_1,v_1)$ and $(u_2,v_2)$ on the time interval $[0,\tau_N)$ and later on, they
represent the processes  $(y_1,z_1)$ and $(y_2,z_2)$. In particular, for $i=1,2$, let us describe the pair $(\bar u_i,\bar v_i) $ such that
\[
\bar u_i  (t) = \bcase u_i(t) & \mbox{ for } 0\le t< \tau_N,\\
y_i (t- \tau_N) & \mbox{ for } \tau_N\le  t \le T,\ecase
\quad
\bar v_i  (t) = \bcase v_i(t) & \mbox{ for } 0\le t< \tau_N,\\
z_i (t - \tau_N) & \mbox{ for } \tau_N\le  t \le T.\ecase
\]
%and
%$$
%\bar u _2  (t) = \bcase u_2 (t) & \mbox{ for } 0\le t< \tau_N,\\
%y_2 (t) & \mbox{ for } \tau_N\le  t \le T,\ecase
%\qquad
%\bar v _2  (t) = \bcase v_2 (t) & \mbox{ for } 0\le t< \tau_N,\\
%z_2 (t) & \mbox{ for } \tau_N\le  t \le T.\ecase
%$$
% such that
%
Note, that  $(\bar u_1,\bar v_1) $ and $(\bar u_2,\bar v_2)$ solve
the truncated equation corresponding  to  \eqref{sysu12.11}-\eqref{sysv12.11} in $[0,\tau_N]$ and equation \eqref{eq1}--\eqref{eq22} in $[\tau_N,T]$.
%that is
%\DEQSZ
%d{u}_i(t)&=& \Big( r_u \Delta u_i (t) %\mbox{div}\big(  \nabla u(t)\big)
%-\chi \mbox{div} \big( u_i(t)\nabla v_i(t)-u^{[q]}(t)\big)\Big)\, dt+\sigma_u u_i(t)dW_1(t)\label{sysu12.111}
%\\
%d{v}_i(t) &=& \big(r_v \Delta v_i(t)+\beta u_i(t) %K(u(t),v(t))
%-\alpha v_i(t)\Big)\, dt +\sigma_v v_i(t) dW_2(t)\,\quad t\in [0,T],\phantom{\big|}\,\,i=1,2.\label{sysv12.111}
%\EEQSZ
%
%\bigskip
%\renewcommand{\tt}{{t}}
%\paragraph{\bf Step II:}
%Before starting, let us give some estimates on $v_1$ and $v_2$. Note, since for any $\delta>1$ we have (see Theorem 2\cite[p. 191]{runst})
%\DEQS
% |\psi_k^{(2)} v|_{H^1_2} \le  |\psi_k^{(2)}|_{H^\delta_2} | v|_{H^1_2}
%\EEQS
%and the embedding $H_2\hookrightarrow H^\delta_2(\CO)$ is a Hilbert-Schmidt - see Assumption \ref{wiener1} -  the multiplication operator $m(v): \phi\mapsto \phi v$ is for $v\in H^1_2(\CO)$ has Hilbert-Schmidt-operator
%$|v|_{H^1_2}$. Secondly, we know $\EE \|u\|_{L^{2\gamma}(0,T;H^\rho_2)}^{2\gamma}<\infty$.
%Hence, it follows by Theorem 4.5 of \cite{vanNeerven1} that
%$\EE \sup_{0\le s\le T}|v(s)|^{2\gamma}_{H^1_2}<\infty$ and, therefore,  $\EE \sup_{0\le s\le T}|v(s)|^{\gamma+1}_{H^1_2}<\infty$.

\paragraph{\bf Step II:}
Our goal is to show that $(u_1,v_1)$ and $(u_2,v_2)$ are identical on the interval $[0,\tau_N]$.
%
%We aim to show that $(u_1,v_1)$ and $(u_2,v_2)$ are identical on the interval $[0,\tau_N]$.
To show this, we  follow an  idea of Gajewski \cite{Gajewski_1994}. Let us consider the function $\phi: \RR \rightarrow \RR$ by
\DEQS
\phi(u) &=& \lk\{ \barray
 u(\ln(u)-1), \quad  &  u> 0,
\\
 0, \quad & u \leq 0.
\earray\rk.
\EEQS
Exploiting the Lemma \ref{lem:gajewski} in Appendix \ref{TL} (see \cite{Gajewski_1994} for more details), we observe that $\phi$ satisfies
\DEQSZ\label{eqn phi gaj}
\phi(u_1)-2\phi\lk( \frac {u_1+u_2}2\rk) +\phi(u_2)\ge \frac 14 \lk(\sqrt{u_1}-\sqrt{u_2}\rk)^2,\quad u_1,u_2\ge 0.
\EEQSZ
Let us consider the functional $\Phi: L^{2}(\CO) \times L^{2}(\CO) \rightarrow \RR$ by
\begin{align}\label{def Phi}
\Phi(u_1,u_2):=\int_\CO \Big\{
\phi(u_1(x))+\phi(u_2(x))-2\phi\lk(\frac {u_1(x)+u_2(x)}2\rk)\Big\}\, dx.
\end{align}
Using \eqref{eqn phi gaj} and \eqref{def Phi}, one can show the following inequality
\begin{align}\label{gaj 1}
\dfrac{1}{4}|\sqrt{u_1}(t)-\sqrt{u_2}(t)|^2_{L^2}\le \Phi(u_1(t),u_2(t)).
\end{align}
Let us now apply the It\^o formula to the functional $\Phi$ for the  process $(u_1(t),u_2(t))$ for  $t \in [0,\tau_N]$.
%In the next lines we apply the It\^o-formula to $\Phi(t)$.
Here, let us first observe that the trace term vanishes.
In particular, we have
% consider first the trace term. Here, we  can observe that this term  vanishes, i.e.\
\begin{align}\label{trace term0}
&\Big\langle D^2 \Phi (u_1,u_2)(s) \begin{pmatrix}
u_1 (s) \\ u_2(s)
\end{pmatrix}, \begin{pmatrix}
u_1(s)  \\ u_2(s)
\end{pmatrix}   \Big\rangle_{L^2}\notag\\
&= \int_{\CO} \Bigg \langle
\begin{pmatrix}
\dfrac{u_2(x,s)}{u_1(x,s)(u_1(x,s)+u_2(x,s))} & \dfrac{-1}{u_1(x,s)+u_2(x,s)}
\\
\dfrac{-1}{u_1(x,s)+u_2(x,s)} & \dfrac{u_1(s,x)}{u_2(x,s)(u_1(x,s)+u_2(x,s))}
\end{pmatrix}, \begin{pmatrix}
u_1(x,s) \\ u_2(x,s)
\end{pmatrix}
   \Bigg\rangle_{L^2}\,dx
   \notag\\
   &= \int_{\CO} \Big\langle
   \begin{pmatrix}
   0 \\ 0
   \end{pmatrix}, \begin{pmatrix}
u_1(s)  \\ u_2(s)
\end{pmatrix}
   \Big\rangle_{L^2}\,dx={\bf 0}.
\end{align}
%% and obtain
%Using \eqref{trace term0} the It\^o-formula gives

%\begin{align}
Now, by applying the It\^o-formula to $\Phi(u_1(t),u_2(t))$ we can write
\DEQS\label{ito gaj}
\lqq{
\Phi((u_1(t),u_2(t))- \Phi(u_1(0),u_2(0)) }
&&
\\ &=&
 \int_0^t \Big[
\la \phi'(u_1(s)), d u_1(s) \ra + \la \phi'(u_2(s)), d u_2(s) \ra
\\ && {}
- \Big\langle \phi' \lk(\frac {u_1(s)+u_2(s)}2\rk), d(u_1+ u_2)(s) \Big \rangle
\Big].
\EEQS
%
%where, in the above, we use the fact that the trace term in the It\^o formula vanishes, i.e.\
%In the next lines we apply the It\^o-formula to $\Phi(t)$.
 %Let us consider first the trace term. Here, we  can observe that this term  vanishes, i.e.\
%
More precisely, we have
\allowdisplaybreaks
\DEQS
%\begin{align}
\label{ito gaj}
\lqq{
\Phi((u_1(t),u_2(t))- \Phi(u_1(0),u_2(0)) }
&&
\\ &=&
 \int_0^t \Big[
\la \ln u_1(s) ,d u_1(s) \ra + \la \ln u_2(s) ,d u_2(s) \ra
- \Big \langle \ln \lk(\frac {u_1(s)+u_2(s)}2\rk), d(u_1(s)+ u_2(s)) \Big \rangle
\Big]
\notag
\\ &=&
 \int_0^t \Big[
\Big \langle \ln \lk(\frac {2u_1(s)}{u_1(s)+u_2(s)} \rk), d u_1(s) \Big \rangle
+ \Big \langle \ln \lk(\frac {2u_2(s)}{u_1(s)+u_2(s)} \rk), d u_2(s) \Big \rangle \Big].
\EEQS
For the sake of clarity, we omit the space variable in the next lines.
 Doing so, we obtain
\DEQS
\lqq{
\Phi((u_1(t),u_2(t))- \Phi(u_1(0),u_2(0))}
&&
\\ &=&\int_0^t \int_{\CO} \Bigg[\Big\{
r_u \Delta u_1 (s) -\chi \Div(u_1(s) \nabla v_1(s))\Big\}
\ln \lk(\frac {2u_1(s)}{u_1(s)+u_2(s)}\rk)
\notag\\ & &
{}+ \Big\{r_u \Delta u_2(s) - \chi\Div(u_2(s) \nabla v_2(s))\Big\}
\ln \lk(\frac {2u_2(s)}{u_1(s) +u_2(s)}\rk)
\Bigg]\,dx\,ds
\notag
\\ & &
{}+\gamma \int_0^t \int_{\CO}
\Big[
u_1(s) \ln \lk( \dfrac{2u_1(s)}{u_1(s)+u_2(s)} \rk)
+ u_2(s) \ln \lk(\dfrac{2u_2(s)}{u_1(s) + u_2(s)} \rk)\Big]\,dx\,ds
\notag
\\ &&
{}+ \int_0^t \int_{\CO}
\Big[u_1 (s)\ln \lk(\frac {2u_1(s)}{u_1(s) + u_2(s)} \rk)
+ u_2(s) \ln \lk(\frac {2u_2(s)}{u_1(s) + u_2(s)} \rk)
\Big]\,d\mathcal W_1(s).
\EEQS
Let us split the sum above into the following three terms
\DEQS
S_1&:= &
\int_0^t \int_{\CO} \Bigg[\Big\{
r_u \Delta u_1 (s) - \chi \Div(u_1(s) \nabla v_1(s)) \Big\}
\ln \lk(\frac {2u_1(s)}{u_1(s)+u_2(s)}\rk)
\notag\\ & &
{}+ \Big\{r_u \Delta u_2(s) -\chi \Div(u_2(s) \nabla v_2(s)) \Big\}
\ln \lk(\frac {2u_2(s)}{u_1(s) +u_2(s)}\rk)
\Bigg]\,dx\,ds,
\\
S_2&:=& \gamma
\int_0^t \int_{\CO}
\Big[
u_1(s) \ln \lk( \dfrac{2u_1(s)}{u_1(s)+u_2(s)} \rk)
+ u_2(s) \ln \lk(\dfrac{2u_2(s)}{u_1(s) + u_2(s)} \rk)\Big]\,dx\,ds,
\EEQS
and
\DEQS
S_3&:=&
 \int_0^t \int_{\CO}
\Big[u_1 (s)\ln \lk(\frac {2u_1(s)}{u_1(s) + u_2(s)} \rk)
+ u_2(s) \ln \lk(\frac {2u_2(s)}{u_1(s) + u_2(s)} \rk)
\Big]\,d\mathcal W_1(s).
\EEQS
%and
%\DEQS
%S_2&:=- \gamma &
% \int_0^t \int_{\CO}
% u_1(s)\Big[u_1 (s)\ln \lk(\frac {2u_1(s)}{u_1(s) + u_2(s)} \rk)
%+ u_2(s) \ln \lk(\frac {2u_2(s)}{u_1(s) + u_2(s)} \rk)
%\Big]\,d\mathcal W_1(s).
%\EEQS
%
%In the next lines we calculate the inner part of $S_2$.
To start with $S_1$, using integration by parts and rearranging the terms, we obtain
\allowdisplaybreaks
\begin{align}\label{int by parts}
&\int_{\CO} \Bigg[
\Big\{
r_u \Delta u_1 (s)- \chi \Div(u_1(s)\nabla v_1(s))\Big\}
 \ln\lk(\frac {2u_1(s)}{u_1(s)+u_2(s)}\rk)
 \notag\\ & \quad
+
\Big\{(r_u \Delta u_2 (s)-\chi \Div(u_2(s)\nabla v_2(s))\Big\} \ln\lk(\frac {2u_2(s)}{u_1(s)+u_2(s)}\rk)\Bigg]\,dx
\notag\\
&=-
\int_{\CO}\Bigg[
\lk(r_u \nabla u_1(s)+ \chi u_1(s)\nabla v_1(s)\rk) \, \frac {u_1(s)+u_2(s)}{2u_1(s)}
\notag\\ & \quad \quad \times
\lk(\frac {2(u_1(s)+u_2(s)) \nabla u_1(s)-2u_1(s)(\nabla u_1(s)+ \nabla u_2(s))}{(u_1(s)+u_2(s))^2}\rk)
\notag\\&\quad
+ \Big( r_u \nabla u_2(s)+\chi u_2(s)\nabla v_2(s)\Big) \, \Big(\frac {u_1(s)+u_2(s)}{2u_2(s)}\Big)
\notag\\
& \quad \times
\lk(\frac {2(u_1(s)+u_2(s))\nabla u_2(s)-2u_2(s) \lk(\nabla u_1(s)+ \nabla u_2(s)\rk)}
{(u_1(s)+u_2(s))^2}\rk)\Bigg]\,dx
\notag\\
&=-
\int_{\CO}\Bigg[
\Big( r_u \nabla u_1(s)+\chi u_1(s)\nabla v_1(s)\Big) \,
\lk(\frac {u_2(s)\nabla u_1(s) - u_1(s)\nabla u_2(s)}{(u_1(s)+u_2(s))u_1(s)}\rk)
\notag\\
&\quad+
\Big(r_u \nabla u_2(s) + \chi u_2(s)\nabla v_2(s) \Big) \,
\Big(\frac {u_1(s)\nabla u_2(s) - u_2(s)\nabla u_1(s)}{(u_1(s)+u_2(s))u_2(s)}\Big)\Bigg]\,dx
\notag\\
&=-
\int_{\CO}\Bigg[
\lk(\frac{r_u \nabla u_1(s)}{u_1(s)}+\chi\nabla v_1(s) \rk) \,\frac {u_1(s) u_2(s)}{u_1(s)+ u_2(s)}
\lk(\frac{\nabla u_1(s)}{u_1(s)} - \frac{\nabla u_2(s)}{u_2(s)}\rk)
\notag\\
&\quad+
\lk(\frac{r_u \nabla u_2(s)}{u_2(s)}+\chi\nabla v_2(s) \rk) \,\frac {u_1(s) u_2(s)}{u_1(s)+ u_2(s)}
\lk(\frac{\nabla u_2(s)}{u_2(s)} - \frac{\nabla u_1(s)}{u_1(s)}\rk)\Bigg]\,dx
\notag\\
& = -\int_{\CO}\Bigg[
\frac{u_1(s) u_2(s)}{u_1(s)+u_2(s)}
\lk(\frac{ \nabla u_1(s)}{u_1(s)}
- \frac{\nabla u_2(s)}{u_2(s)}\rk)
\Big\{r_u \lk(\frac{\nabla u_1(s)}{u_1(s)}
-\frac{\nabla u_2(s)}{u_2(s)}\rk)
\notag\\
&\quad+ \chi \big(\nabla v_1(s)-\nabla v_2(s)\big)\Big\}
\Bigg]\,dx
\notag\\
&=-\int_{\CO}\Bigg[\frac{u_1(s) u_2(s)}{u_1(s)+u_2(s)}
\Big\{ r_u \Big|\nabla \ln\Big( \frac{ u_1(s)}{u_2(s)}\Big)\Big|^2
\notag\\
&\quad
+\chi \nabla \ln\Big( \frac{ u_1(s)}{u_2(s)} \Big)\cdot
 \nabla \Big(v_1(s)- v_2(s) \Big)\Big\}\Bigg]\,dx.
%\notag\\
%&\leq -r_u (1 - \eps) \int_{\CO}
%\frac {u_1(s) u_2(s)}{u_1(s)+u_2(s)} \lk|\nabla \ln\lk( \frac{ u_1(s)}{u_2(s)}\rk)\rk|^2\,dx
%\notag\\
%&\quad
%+ \frac{\chi^2}{4r_u \eps} \int_{\CO}
%\frac {u_1(s) u_2(s)}{u_1(s)+u_2(s)} \big|\nabla (v_1(s)-v_2(s)) \big|^2\,dx.
\end{align}
Using the Young inequality we can show for any  $\eps>0$, that  the left hand side of \eqref{int by parts} is dominated by
\allowdisplaybreaks
\begin{align}\label{int by parts eps}
& -r_u (1 - \eps) \int_{\CO}
\frac {u_1(s) u_2(s)}{u_1(s)+u_2(s)} \lk|\nabla \ln\lk( \frac{ u_1(s)}{u_2(s)}\rk)\rk|^2\,dx
\notag\\
&\quad
+ \frac{\chi^2}{4r_u \eps} \int_{\CO}
\frac {u_1(s) u_2(s)}{u_1(s)+u_2(s)} \big|\nabla (v_1(s)-v_2(s)) \big|^2\,dx.
\end{align}
Using $\frac{u_1 u_2}{u_1+u_2} \leq u_2$, we observe that \eqref{int by parts eps}
can be further simplified as
\begin{align}\label{esti u2 sim}
& -r_u (1- \eps) \int_{\CO}
\frac {u_1(s) u_2(s)}{u_1(s)+u_2(s)} \lk|\nabla \ln\lk( \frac{ u_1(s)}{u_2(s)}\rk)\rk|^2\,dx
\notag\\
&\quad
+ \frac{\chi^2}{4r_u \eps} \int_{\CO}
 u_2(s) \big|\nabla (v_1(s)-v_2(s)) \big|^2\,dx.
\end{align}
We now evaluate the second term of \eqref{esti u2 sim}, i.e.\ $\int_{\CO} u_2 \lk|\nabla v_1-\nabla v_2\rk|^2 \,dx$.
Applying  the H\"older inequality  with $\frac{1}{2}+\frac{1}{2}=1$ and using the definition of stopping time, we achieve %choosing \color{red}{$p'=\frac{2}{1-s_2}$} we achieve
%\DEQSZ\label{vterm}\nonumber
%\lqq{
%\int_{\CO} u_2 \lk|\nabla v_1-\nabla v_2\rk|^2 \,dx}
%&&
%\\
%&
%\leq &\lk|\nabla v_1-\nabla v_2\rk|^2_{L^{\frac{2(q+1)}{q}}} |u_2|_{L^{q+1}}
%%\leq  C\,  \lk|\nabla v_1-\nabla v_2\rk|^2_{L^{2p}} |u_2(s)|_{H^{s_2}_2}
%\leq N\,  \lk|\nabla v_1-\nabla v_2\rk|^2_{L^{2}}
%.
%\EEQSZ
\begin{align}\label{vterm}
\int_{\CO} u_2 \lk|\nabla v_1-\nabla v_2\rk|^2 \,dx
&
\leq \lk|\nabla v_1-\nabla v_2\rk|^2_{L^{2}} |u_2|_{L^{2}}
%\leq  C\,  \lk|\nabla v_1-\nabla v_2\rk|^2_{L^{2p}} |u_2(s)|_{H^{s_2}_2}
\leq N\,  \lk|\nabla v_1-\nabla v_2\rk|^2_{L^{2}}
.
\end{align}
To evaluate $S_2$, we use Claim \ref{claim nonneg}.
 We first note that if $u_1=u_2=0$, then
\DEQS
S_2= \gamma
\int_0^t \int_{\CO}
\Big[
u_1(s) \ln \lk( \dfrac{2u_1(s)}{u_1(s)+u_2(s)} \rk)
+ u_2(s) \ln \lk(\dfrac{2u_2(s)}{u_1(s) + u_2(s)} \rk)\Big]\,dx\,ds
=0.\notag
\EEQS
Otherwise, we note that $S_2$ can be written as
%\DEQS
%S_2 &=&  \lk\{ \barray
%\mathlarger{\int_0^t \int_{\CO}}
%\gamma\Big[
%u_1(s)\Big\{ \ln \Big( \dfrac{2}{1 + \frac{u_2(s)}{u_1(s)}} \Big)
%+ \lk( \frac{u_2(s)}{u_1(s)} \rk)
%\ln \Big(
%\dfrac{ 2 \lk(\frac{u_2(s)}{u_1(s)}\rk)}
%{1 + \frac{u_2(s)}{u_1(s)}} \Big)\Big\}\Big]\,dx\,ds, \quad \mbox{if} \quad u_1(s) \neq 0,
%%&\dot{u}(t)  \mbox{div}( D_u(u(t),v(t)) \nabla u(t)+\mbox{div} \chi(u(t),v(t)) u(t)\nabla v(t)+H(u(t),v(t))\, \phantom{\Big|}
%\\
%\mathlarger{ \int_0^t \int_{\CO}}
%\gamma \Big[
%u_2(s)\Big\{ \frac{u_1(s)}{u_2(s)}
%\ln \Big(
%\dfrac{ 2 \lk(\frac{u_1(s)}{u_2(s)}\rk)}
%{1 + \frac{u_1(s)}{u_2(s)}}\Big)
%+
%\ln \Big(\dfrac{2}{1 + \frac{u_1(s)}{u_2(s)}}\Big) \Big\}
% \Big]\,dx\,ds, \quad \mbox{if} \quad u_2(s) \neq 0.
%\phantom{\big|}
%\earray\rk.
%\EEQS
%In other way, we can write
\DEQSZ\label{S2}
S_2 &=&\lk\{ \barray
\mathlarger{\gamma \int_0^t \int_{\CO}}
u_1(s) f(u)\,dx\,ds, \quad \mbox{if} \quad u_1(s) \neq 0, \,\,u=\dfrac{u_2}{u_1},
%&\dot{u}(t)  \mbox{div}( D_u(u(t),v(t)) \nabla u(t)+\mbox{div} \chi(u(t),v(t)) u(t)\nabla v(t)+H(u(t),v(t))\, \phantom{\Big|}
\\
\mathlarger{\gamma \int_0^t \int_{\CO}}
u_2(s)f(u)\,dx\,ds, \quad \mbox{if} \quad u_2(s) \neq 0,\,\, u=\dfrac{u_1}{u_2},
\phantom{\big|}
\earray\rk.
\EEQSZ
where $f$ is given in Claim \ref{claim nonneg}.
Putting $u=\dfrac{u_2}{u_1}$ in \eqref{def f} in Claim \ref{claim nonneg}, we see that
 \begin{align}
  S_2
\leq \gamma \int_0^t \int_{\CO} |\sqrt{u_1(x,s)}-\sqrt{u_2(x,s)}|^2_{L^2}\,dx\,ds.
 \end{align}
Furthermore, since $S_3$ is a local martingale, the  expectation of this term vanishes.
Combining all the estimates above, we obtain from \eqref{gaj 1}
\begin{align}\label{combine esti}
&\dfrac{1}{4}|\sqrt{u_1}(t)-\sqrt{u_2}(t)|^2_{L^2}
+ r_u(1-\ep)\int_0^t \int_{\CO}
\frac{u_1 (s) u_2 (s)}{u_1(s) + u_2(s)} \,\Big|\nabla \ln \lk( \frac{ u_1(s)}{u_2(s)}\rk)\Big|^2\,dx\,ds
\notag\\
&\leq \Phi(u_1(0),u_2(0))+\frac{\chi^2C_{\CO}N}{4r_u \eps}
 \,\int_0^t \lk|\nabla v_1(s)-\nabla v_2(s) \rk|^2_{L^{2}}\,ds
 \notag\\ & \quad
+ \gamma \int_0^t \int_{\CO} |\sqrt{u_1}(s)-\sqrt{u_2}(s)|^2_{L^2}\,dx\,ds
.
\end{align}
Since $ L^1(\CO) \hookrightarrow H_{2}^{-1}(\CO)$ in one dimension, one can show that
 there exists some  constant $C>0$ such that
\begin{align*}
\EE \| \nabla v_1-\nabla v_2\|_{L^2(0,T;L^{2})}^2
\le C \EE \| u_1-u_2\|_{L^2(0,T;H^{-1}_2)}^2.
\end{align*}
By elementary calculations and using the embedding $L^{1}(\CO) \hookrightarrow H^{-1}_2(\CO)$, we obtain
\begin{align*}
& \EE\Big[ \int_0^T | u_1(s)-u_2(s)|^2_{H^{-1}_2}\,ds\Big]
\leq \EE\Big[ \int_0^T | u_1(s)-u_2(s)|^2_{L^{1}}\,ds\Big]
\\ & \leq
 \EE\Big[ \int_0^ T \Big|(\sqrt{u_1}(s)-\sqrt{u_2}(s)) (\sqrt{u_1}(s)+\sqrt{u_2}(s))\Big|^2_{L^{1}}\,ds\Big]
\\
& \le \EE\Big[ \int_0^ T \Big(
\int_{\CO}|\sqrt{u_1}(x,s)-\sqrt{u_2}(x,s)|\,|\sqrt{u_1}(x,s)+\sqrt{u_2}(x,s)|\,dx \Big)^{2}\,ds\Big]
\\
& \le C\, \EE\Big[ \int_0^T \Big( \int_{\CO}|\sqrt{u_1}(x,s)-\sqrt{u_2}(x,s)|^{2}\,dx \Big)
 \int_{\CO}|\sqrt{u_1}(x,s)+\sqrt{u_2}(x,s)|^{2}\,dx \,ds\Big]
\\ & \le \EE\Big[\int_0^T   |\sqrt{u_1}(s)-\sqrt{u_2}(s)|^{2}_{L^2}\,
 |u_1(s)+u_2(s)|_{L^{1}}\,ds \Big].
%\\ & \leq
%2C\,N  \, \EE\Big[ \int_0^ T |\sqrt{u_1}(s)-\sqrt{u_2}(s)|^2_{L^2}\,ds \Big]
% .
%+  C(\ep_1)\EE\Big[ \int_0^ T |\sqrt{u_1}(s)-\sqrt{u_2}(s)|^2_{L^2} \, ds\Big].
\end{align*}
Then,  by the definition of the stopping time we know $|u_i|_{L^{1}} \leq N$ for $i=1,2$, and, hence
\begin{align*}
& \EE\Big[ \int_0^T | u_1(s)-u_2(s)|^2_{H^{-1}_2}\,ds\Big]
\leq
2C\,N  \, \EE\Big[ \int_0^ T |\sqrt{u_1}(s)-\sqrt{u_2}(s)|^2_{L^2}\,ds \Big]
 .
%+  C(\ep_1)\EE\Big[ \int_0^ T |\sqrt{u_1}(s)-\sqrt{u_2}(s)|^2_{L^2} \, ds\Big].
\end{align*}
Substituting above in \eqref{combine esti}, we obtain
\begin{align}
&\dfrac{1}{4}\EE\lk[|\sqrt{u_1}(t)-\sqrt{u_2}(t)|_{L^2}^2 \rk]
+r_u (1- \eps)\EE \lk[\int_0^T \int_{\CO} \frac{u_1(s) u_2(s)}{u_1(s) + u_2(s)}
\lk|\nabla \ln \lk(\frac {u_1(s)}{u_2(s)}\rk)\rk|^2\,dx\,ds \rk]
\notag\\
&
\le \Phi(u_1(0),u_2(0))+\lk(\frac{\chi^2C\,C_{\CO}N^2}{2r_u \eps}+\gamma \rk)
  \, \EE\lk[ \int_0^ T |\sqrt{u_1}(s)-\sqrt{u_2}(s)|^2_{L^2}\,ds \rk]
.
\end{align}
Using the fact that for $u_1(0)=u_2(0)$ we have $\Phi(u_1(0),u_2(0))=0$, a  Gronwall argument gives
\begin{align}
\EE \Big[ |\sqrt{u_1(t)}-\sqrt{u_2(t)}|^2_{L^2}
\Big] \le e^{2\big(\frac{\chi^2C\,C_{\CO}N^2}{2r_u \eps}+1 \big)T} \Phi(u_1(0),u_2(0))=0.
\end{align}

\paragraph{\bf Step IV:}
We show that $\PP\lk( \tau_N<T\rk) \to 0$ as $N\to\infty$.
\begin{align*}
\{\tau_N < T \} \subset \Big\{ &\sup_{s \in [0,T]} |u_1(s)|_{L^{2}} \ge N
\quad \hbox{or} \quad
 \sup_{s \in [0,T]} |u_2(s)|_{L^{2}} \ge N \quad \hbox{or} \quad
  \\ &\sup_{s \in [0,T]} |u_1(s)|_{L^{1}} \ge N 
 \quad \hbox{or} \quad
 \sup_{s \in [0,T]} |u_2(s)|_{L^{1}} \ge N \Big\}.
\end{align*}
Therefore, due to Proposition \ref{uniformboundtaun}, Theorem 2.6 of \cite{EH+DM+TT_2019} and, since,  $\hbox{LlogL}(\CO) \hookrightarrow L^1(\CO)$,
one can observe using the Chebyscheff inequality
\begin{align*}
\PP \lk(\tau_N < T \rk) &\leq \PP\Big( \sup_{s \in [0,T]} |u_1(s)|_{L^{2}} \ge N \Big)
+ \PP \Big( \sup_{s \in [0,T]} |u_2(s)|_{L^{2}} \ge N \Big)
\\
&\quad + \PP\Big( \sup_{s \in [0,T]} |u_1(s)|_{L^{1}} \ge N \Big)
+ \PP \Big( \sup_{s \in [0,T]} |u_2(s)|_{L^{1}} \ge N \Big)
\leq \dfrac{4C}{N}.
\end{align*}
It follows $$\PP\lk( \tau_N \le   T \rk)\longrightarrow 0 \quad \mbox{as} \quad N\to\infty.$$
Hence, both processes $u_1$ and $u_2$ coincide on $[0,T]$, so are $v_1$ and $v_2$.
This completes the proof of Proposition \ref{thm path uniq}.
\end{proof}
%\end{proof}

%\section{Existence of unique strong solution to system \eqref{chemonoise}}\label{martingale corr}
\section{Proof of Proposition \ref{uniformboundtaun}}\label{add reg proof}

\begin{proof}[Proof of Proposition \ref{uniformboundtaun}]
%{\color{red}
The proof consists of two steps. In the first step, we will estimate the $p$-th moment of $L^1$ norm of $u$ for $p>1$. In the next step, 
%we obtain $p$-th moment of $H^1_2$ norm of $v$ for $p>1$. In the final step, 
we obtain uniform bound for the $L^{2}(\Omega;L^{\infty}([0,T];L^{2}(\CO))$ norm of $u$ and the $L^{2}(\Omega;L^\infty([0,T];H^{1}_2(\CO))$ norm of $v$.
\begin{steps}
\item Let us define the integral operator $\mathcal I:L^1(\CO) \rightarrow \RR$ by $\mathcal I(u):=\int_{\CO} u(x)dx$. Then, $\mathcal I$ is linear functional, hence is of class~$C^2$; 
(see \cite[p.\ 11 Example 1.3 (b)]{Ambrosetti_1995}).
Applying the It\^o formula \cite[Chapter 4.4]{pratozab}) to the functional $\mathcal I(u)=\int_{\CO} u(x)dx$ for the process $\{u(t)  \}_{t \in [0,T]}$, we obtain
\begin{align}
\mathcal I(u(t))& = \mathcal I(u(0)) + \int_0^t \Big\langle \dfrac{\partial \mathcal I}{\partial u}, d u_1(s)  \Big\rangle
+ \frac{1}{2} \int_0^t \hbox{Tr} \big[(\dfrac{\partial^2 \mathcal I}{\partial u^2})(u u^\top) \big]\,ds.
\end{align}
Using $\dfrac{\partial^2 \mathcal I}{\partial u^2}=0$, the above equation reduces to
\begin{align}\label{eqn Phi1}
\mathcal I(u(t)) &= \mathcal I(u(0)) + \int_0^t \int_{\CO} \Big[ r_u \Delta u(s,x) - \chi \di(u(s,x)\nabla v(s,x)) + \gamma u(s,x) \Big]\,dx\,ds
\notag\\
& \quad + \int_0^t \int_{\CO} u(s,x)\, d\mathcal W_1(s,x).
\end{align}
Now using integration by parts in space variable and by Neumann boundary conditions, we see that the second and third terms in the right hand side of \eqref{eqn Phi1} vanishes. Indeed, we have the following equality.
\begin{align}
r_u \int_{\CO} \Delta u(s,x)\,dx & =
 r_u \int_{\partial \CO} \nabla u(s,x) \underbrace{\frac{\partial u}{\partial \nu}(s,x)}_{=0}\,d\mathcal{S}(x)
=0,
\label{int by parts1}\\
 \chi \int_{\CO} \di(u(s,x) \nabla v(s,x))\,dx & =
 \chi \int_{\partial \CO} u(s,x) \underbrace{\frac{\partial v}{\partial \nu}(s,x)}_{=0}\,d\mathcal{S}(x)
=0,\label{int by parts2}
\end{align}
where, $\mathcal{S}$ is the surface measure on $\partial \CO$.
In this way, substituting \eqref{int by parts1}-\eqref{int by parts2} in \eqref{eqn Phi1}, we obtain
\begin{align}\label{det iu}
\mathcal I(u(t)) = \mathcal I(u(0)) + \gamma \int_0^t \mathcal I(u(s))\, ds + \int_0^t \int_{\CO} u(s,x)
\, d \mathcal{W}_1(s,x).
\end{align}
Our aim in this step is to obtain the higher moment of $\mathcal I(u(t))$ for $t \in [0,T]$. 
% To do so, for $p>1$, we aim to bound $\EE\Big[ \sup_{0 \leq s \leq t} |u(s)|^p_{L^1}\Big]$ by $C$. 
Using \eqref{sum W1}, and $\{\psi_k\}_{k \in \ZZ}$ being orthonormal basis in $L^2(\CO)$, it is clear that
 $\{\psi^{(1)}_k\}_{k \in \ZZ}=\{\lambda_k^{(1)} \psi_k\}_{k \in \ZZ}$ is an orthonormal basis of $\CH_1$. Therefore, we have
%$\{\psi_k\}_{k \in \ZZ}$ being orthonormal basis in $L^2(\CO)$ and using \eqref{def Wi}, we have
\begin{align}\label{sum W1}
\mathcal{W}_1(r,x)=\sum_{k \in \ZZ} \psi_k^{(1)}(x) \beta_k^{(1)}(r)=\sum_{k \in \ZZ}\lambda_k^{(1)} \psi_k(x) \beta_k^{(1)}(r).
\end{align} 
We define 
\begin{align}\label{def Phiu}
\Phi^u_s : \CH_1 \rightarrow \mathbb{R} \quad \hbox{by} \quad
\Phi^u_s(\eta)= \int_{\CO} u(s,x)\eta(x)\, dx, \quad \hbox{for all} \,\, \eta \in \CH_1.
\end{align}
Let $\Phi^u= \Big\{ \Phi^u_t, \,\, t \in [0,T]  \Big\}$ be a measurable $L_2^0$-valued process. We define the norm of $\Phi^u$ by
\begin{align}
\|\Phi^u\|_T := \Big[ \EE \Big(\int_0^T  \|\Phi^u_s   \|_{L_2^{0}}^2 \, ds  \Big)   \Big]^{\frac{1}{2}}.
\end{align}
We refer to \cite[Chapter 4] {pratozab} to define the stochastic integral with respect to $\mathcal{W}$ of any $L_2^{0}$-valued predictable process $\Phi^u$ such that  $\|\Phi^u\|_T < \infty$.  
Following equation (3.7) in \cite{Dalang+Lluis} with $H=\RR$ and $V=\CH_1$ , we now compute $\|\Phi^u\|_T^2$.
\begin{align}\label{Phi esti}
\|\Phi^u\|_T^2&=\EE\Big[\int_0^T \|\Phi^u_s   \|_{L_2^{0}}^2 \, ds   \Big]
= \EE\Big[\int_0^T \sum_{k \in \ZZ} |\Phi^u_s (\psi_k^{(1)}) |^2 \, ds   \Big]
\notag\\
&=\EE\Big[\int_0^T \sum_{k \in \ZZ} \Big|\int_{\CO} u(s,x) \lambda_k^{(1)} \psi_k(x)\,dx \Big|^2 \, ds   \Big]
\notag\\
&=\EE\Big[\int_0^T \sum_{k \in \ZZ} |\lambda_k^{(1)}|^2 
 \Big|\int_{\CO} u(s,x) \psi_k(x)\,dx \Big|^2 \, ds   \Big]
 \notag\\
& \leq \EE\Big[\int_0^T \sum_{k \in \ZZ} |\lambda_k^{(1)}|^2 |\psi_k|_{L^\infty}^2
 \Big|\int_{\CO} u(s,x)\,dx \Big|^2 \, ds   \Big].
\end{align}
Now we use Monotone Convergence Theorem (Theorem 1.2.7 of \cite{Walter1987}, Theorem 2.15 of \cite{Folland_1999}). To do so, we take $X=\Omega \times [0,T]$, $f_k(\omega,s):=  |\lambda_k^{(1)}|^2 |\psi_k|_{L^\infty}^2 \Big|\int_{\CO} u(s,x,\omega)\,dx \Big|^2 $. Then, we have
\begin{align*}
\int_{\Omega \times [0,T]} \sum_{k \in \ZZ} f_k(\omega,s)\,ds\,d \mathbb{P}(\omega)
= \sum_{k \in \ZZ} \int_{\Omega \times [0,T]}  f_k(\omega,s)\,ds\,d \mathbb{P}(\omega),
\end{align*}
which implies the following equality
\begin{align}\label{mct}
 \EE\Big[\int_0^T \sum_{k \in \ZZ} |\lambda_k^{(1)}|^2 |\psi_k|_{L^\infty}^2
 \Big|\int_{\CO} u(s,x)\,dx \Big|^2 \, ds   \Big]
 \notag\\
= \sum_{k \in \ZZ} |\lambda_k^{(1)}|^2 |\psi_k|_{L^\infty}^2 
\EE\Big[\int_0^T  
 \Big|\int_{\CO} u(s,x)\,dx \Big|^2 \, ds   \Big]
 .
\end{align}
Combining \eqref{Phi esti} and \eqref{mct}, using Theorem 2.6 of \cite{EH+DM+TT_2019} and, since,  $\hbox{LlogL}(\CO) \hookrightarrow L^1(\CO)$, we achieve
\begin{align}\label{Phi esti 1}
\|\Phi^u\|_T^2& \leq
\sum_{k \in \ZZ} |\lambda_k^{(1)}|^2 |\psi_k|_{L^\infty}^2 
\EE\Big[\int_0^T  
 \Big|\int_{\CO} u(s,x)\,dx \Big|^2 \, ds   \Big]
 < \infty.
\end{align}
Now, by Proposition 3.4 of \cite{Dalang+Lluis} with $H=\mathbb{R}$ and $f_k=1$, we have
\begin{align}\label{inter 1}
\int_0^T \Phi^u_s \, d \mathcal{W}_1(s)& = \sum_{j=1}^{\infty} \int_0^T \Phi^u_s(\psi_j^{(1)})\, d \beta_j^{(1)}(s)
= \sum_{j \in \ZZ} \lambda_j^{(1)} \int_0^T 
\Big(\int_{\CO} u(s,x) \psi_j(x)\,dx \Big) d \beta_j^{(1)}(s).
\end{align}
Also, using \eqref{def Phiu}, we have
\begin{align}\label{inter 2}
\int_0^T \Phi^u_s \, d \mathcal{W}_1(s)& 
= \int_0^T \Big(\int_{\CO}
\sum_{j\in \ZZ} u(s,x) \lambda_j^{(1)} \psi_j(x)\,dx \Big) d \beta_j^{(1)}(s).
\end{align}
Using \eqref{inter 1} and \eqref{inter 2}, we now estimate the following
\begin{align}
&\EE\Big[\sup_{0 \leq s \leq t}\Big| \int_0^s \int_{\CO} \sum_{k \in \ZZ}
  \psi_k^{(1)}(x)\, u(r,x) \, d x \, d \beta_k^{(1)}(r)  \Big|   \Big]
\notag\\ 
&= \EE\Big[\sup_{0 \leq s \leq t}\Big|\sum_{k \in \ZZ} \int_0^s 
\Big(\int_{\CO}  u(r,x)\,\psi_k^{(1)}(x)\, d x \Big)\, d \beta_k^{(1)}(r)  \Big|   \Big]
\notag\\
&\leq \EE\Big[\sup_{0 \leq s \leq t}\sum_{k \in \ZZ} 
\Big|\int_0^s  \Big(\int_{\CO}  u(r,x)\,\psi_k^{(1)}(x)\, d x \Big)\, d \beta_k^{(1)}(r)  \Big|   \Big]
\notag\\
&= \EE\Big[\sum_{k \in \ZZ} \sup_{0 \leq s \leq t}
\Big|\int_0^s  \Big(\int_{\CO}  u(r,x)\,\psi_k^{(1)}(x)\, d x \Big)\, d \beta_k^{(1)}(r)  \Big|   \Big].
\end{align}
Again, we use Monotone Convergence Theorem (Theorem 1.2.7 of \cite{Walter1987}, Theorem 2.15 of \cite{Folland_1999}). To do so, we take $$X=\Omega,
\quad 
 f_k(\omega,t):= \sup_{0 \leq s \leq t}
\Big|\int_0^s  \Big(\int_{\CO}  u(r,x)\,\psi_k^{(1)}(x)\, d x \Big)\, d \beta_k^{(1)}(r)  \Big| 
.$$ Then, we have
\begin{align*}
\int_{\Omega} \sum_{k \in \ZZ} f_k(\omega,t)\,d \mathbb{P}(\omega)
= \sum_{k \in \ZZ} \int_{\Omega }  f_k(\omega,t)\,d \mathbb{P}(\omega),
\end{align*}
which implies the following equality
\begin{align}\label{MCT}
&\EE\Big[\sum_{k \in \ZZ} \sup_{0 \leq s \leq t}
\Big|\int_0^s  \Big(\int_{\CO}  u(r,x)\,\psi_k^{(1)}(x)\, d x \Big)\, d \beta_k^{(1)}(r)  \Big|   \Big]
\notag\\
& = \sum_{k \in \ZZ} \EE\Big[ \sup_{0 \leq s \leq t}
\Big|\int_0^s  \Big(\int_{\CO}  u(r,x)\,\psi_k^{(1)}(x)\, d x \Big)\, d \beta_k^{(1)}(r)  \Big|   \Big].
\end{align}
Using the Burkholder-Davis-Gundy inequality for real-valued Brownian motion, we achieve
\begin{align}\label{bdg k}
& \EE\Big[ \sup_{0 \leq s \leq t}
\Big|\int_0^s  \Big(\int_{\CO}  u(r,x)\,\psi_k^{(1)}(x)\, d x \Big)\, d \beta_k^{(1)}(r)  \Big|   \Big]
\notag\\
&\leq \EE\Big[ 
\int_0^t  \Big(\int_{\CO}  u(r,x)\,\psi_k^{(1)}(x)\, d x \Big)^2 \, dr  \Big]^{\frac{1}{2}}\notag\\
& \leq |\psi_k^{(1)}|^2_{L^\infty}
 \EE\Big[ 
\int_0^s  \Big(\int_{\CO}  u(r,x)\, d x \Big)^2 \, dr  \Big]^{\frac{1}{2}}.
\end{align}
Thus, summing over all $k$ and using \eqref{mct} and \eqref{bdg k} we have 
\begin{align}\label{MCT-BDG}
&\EE\Big[\sum_{k \in \ZZ} \sup_{0 \leq s \leq t}
\Big|\int_0^s  \Big(\int_{\CO}  u(r,x)\,\psi_k^{(1)}(x)\, d x \Big)\, d \beta_k^{(1)}(r)  \Big|   \Big]
\notag\\
& = \sum_{k \in \ZZ} \EE\Big[ \sup_{0 \leq s \leq t}
\Big|\int_0^s  \Big(\int_{\CO}  u(r,x)\,\psi_k^{(1)}(x)\, d x \Big)\, d \beta_k^{(1)}(r)  \Big|   \Big]
\notag\\
& \leq  \sum_{k \in \ZZ} |\psi_k^{(1)}|^2_{L^\infty}
 \EE\Big[ 
\int_0^t  \Big(\int_{\CO}  u(r,x)\, d x \Big)^2 \, dr  \Big]^{\frac{1}{2}}.
\end{align}
Using the similar argument for $p>1$, we obtain
%\leq C_p \EE\Big[\int_0^t \|u(s)\|_{L_{HS}(\CH_1,L^2)}^2 \, ds \Big]^{\frac{p}{2}}
%\notag\\
%&
\begin{align}\label{BDG W1}
&\EE\Big[\sup_{0 \leq s \leq t}\Big| \int_0^s \int_{\CO} \sum_{k \in \ZZ}
  \psi_k^{(1)}(x) u(r,x) \, d x \, d \beta_k^{(1)}(r)  \Big|^p   \Big]
%\notag\\
%&  \leq C_p \EE\Big[\Big(\int_0^t \sum_{k \in \ZZ} 
%\Big(  \int_{\CO}   u(r,x)\psi_k^{(1)}(x) \,dx \Big)^2 \, ds  \Big)^{\frac{p}{2}}\Big]
\notag\\
& \leq C_p \gamma_1^{\frac{p}{2}} \,
\EE \Big[\Big(\int_0^t \Big(  \int_{\CO}   u(r,x) \,dx \Big)^2\,ds \Big)^{\frac{p}{2}}\Big],
\end{align}
where $\gamma_1= \sum_{k \in \ZZ} |\psi_k^{(1)}|_{L^\infty}^2 \leq \sum_{k \in \ZZ} |\lambda^{(1)}_k |^2 |\psi_k|_{L^\infty}^2 \leq \sum_{k \in \ZZ} |\lambda^{(1)}_k |^2 < \infty$.
%\begin{align}\label{BDG W1}
%\EE\Big[\sup_{0 \leq s \leq t}\Big| \int_0^s \int_{\CO}  u(r,x) \, d \mathcal{W}_1(r,x)  \Big|^p   \Big]
%&=\EE\Big[\sup_{0 \leq s \leq t}\Big| \int_0^s \int_{\CO} u(r,x) \,\sum_{k=1}^\infty  \psi_k^{(1)}(x)\,dx \, d\beta_k^{(1)}(r)  \Big| ^p  \Big]
%\notag\\
%& \leq C_p \EE\Big[\Big(\int_0^t \|u(r)\|^2_{L_{\text{HS}}(\CH_1,H^{-1}_2)}\,dr \Big)^{\frac{p}{2}}\Big].
%\end{align}  
%Let us assume that $u$ is Hilbert-Schmidt\footnote{$L_{\text{HS}}(H_1,H_2)$ denotes the space of linear operators
%from~$H_1$ to~$H_2$ equipped with the Hilbert--Schmidt norm.} operator from $\CH_1 \hookrightarrow H^{-1}_2(\CO)$.  
%Since $\{\psi_k\}_{k \in \ZZ}$ is an orthonormal basis in $L^2(\CO)$, for any $v \in $
%Again using $\{\psi^{(1)}_k\}_{k \in \ZZ}$ is an orthonormal basis of $\CH_1$ , ,
%Let us remind the abbreviation $\gamma_1=\sum_{k\in\ZZ}\,(\lambda^{(1)}_k)^2$ (see \eqref{defofgamma}) and $\mu_k=(2 \pi k)^2$ as in Remark \ref{rem wiener}. Using Proposition 1 in \cite{Barbu+Roeckner} in $d=1$, we have
%\begin{align}\label{esti jk}
%\|u(r)\|^2_{L_{\text{HS}}(\CH_1,H^{-1}_2)} 
%&= \sum_{k \in \ZZ} |u(r)\lambda_k^{(1)}\psi_k|_{H^{-1}_2}^2
%= \sum_{k \in \ZZ} (\lambda_k^{(1)})^2 | u(r)\psi_k|_{H^{-1}_2}^2
%= \sum_{k \in \ZZ \atop k \neq 0} \dfrac{(\lambda_k^{(1)})^2}{\mu_k} | u(r)\mu_k^{1/2} \psi_k|_{H^{-1}_2}^2
%\notag\\
%&\leq  C_1 \sum_{k \in \ZZ } (\lambda_k^{(1)})^2 | u(r)|_{H^{-1}_2}^2
%\leq  C_1 \sum_{k \in \ZZ} (\lambda_k^{(1)})^2 \Big|  \int_{\CO} u(r,x) \,dx\Big|^2
%\notag\\
%&\leq C_1 \gamma_1 \Big|  \int_{\CO} u(r,x) \,dx\Big|^2.
%\end{align}
Using \eqref{BDG W1} we have
\begin{align}
&\EE\Big[\sup_{0 \leq s \leq t}\Big| \int_0^s \int_{\CO}  u(r,x) \, d \mathcal{W}_1(r,x)  \Big|^p   \Big]
\leq  C_p \gamma_1^{\frac{p}{2}} \, \EE\Big[ \Big(\int_0^t \Big( \int_{\CO} u(s,x) dx  \Big)^2 \, ds   \Big)^{\frac{p}{2}}\Big]
\notag\\
&\leq  C_p \gamma_1^{\frac{p}{2}} \, \EE\Big[ \Big(\sup_{0 \leq s \leq t} \int_{\CO} u(s,x)dx \Big)^{\frac{p}{2}}
\Big(\int_0^t \int_{\CO} u(s,x) dx \, ds \Big)^{\frac{p}{2}}    \Big]
\notag\\
&
\leq \frac{1}{2}\EE\Big[ \sup_{0 \leq s \leq t} \Big(\int_{\CO} u(s,x)dx \Big)^p\Big]
+ \frac{1}{2}t^{p-1} C_p \gamma_1^{\frac{p}{2}} \, \EE\Big[ \int_0^t \Big(\int_{\CO} u(s,x)dx \Big)^p\,ds\Big].
\end{align}
We now raise power $p \geq 1$ on both sides of \eqref{det iu}, take supremum over $s \in [0,T]$, and then the expectation. Hence, we obtain
\begin{align}
\EE\Big[ \sup_{0 \leq s \leq t} \mathcal{I}^p(u(t))\Big]
&= \EE[  \mathcal{I}^p(u(0))]
+ \gamma \,p\, \EE\Big[ \Big( \int_0^t \mathcal{I}(u(s))\,ds \Big)^p\Big]
+\frac{1}{2}\EE\Big[ \sup_{0 \leq s \leq t} \mathcal{I}^p(u(s))\Big]
\notag\\ &
\quad + \frac{1}{2}t^{p-1} C_p \gamma_1^{\frac{p}{2}} \,\EE\Big[ \int_0^t  \mathcal{I}^p(u(s))\,ds\Big].
\end{align}
This implies
\begin{align}
\frac{1}{2}\EE\Big[ \sup_{0 \leq s \leq t} \mathcal{I}^p(u(s))\Big]
&\leq \EE[\mathcal{I}^p(u(0))]
+ (\gamma \,p+ \frac{ C_p \gamma_1^{\frac{p}{2}}}{2})t^{p-1}\EE\Big[ \int_0^t \mathcal{I}^p(u(s))\,ds\Big].
\end{align}
Using the Gronwall inequality we obtain
\begin{align}\label{eqn uL1}
\EE\Big[ \sup_{0 \leq s \leq t} \mathcal{I}^p(u(s))\Big]
&\leq \EE[   \mathcal{I}^p(u(0))]e^{(\gamma \,p+ \frac{ C_p \gamma_1^{\frac{p}{2}} }{2})T^{p}}.
\end{align}
%In general, we cannot apply It\^o formula to the functional $\mathcal I_1:L^1(\CO) \rightarrow \RR$ by $\mathcal \mathcal I_1(u)=|u|_{L^1}$. For us, we have applied It\^o formula to the functional $\mathcal I(u(t))=\int_{\CO} u(t,x)\,dx$ which is of class~$C^2$.
Since,  $u$, the cell density of the chemotaxis system is non-negative, 
$$ u(x,t) \geq 0 \,\, \mbox{for  a.e.\,\,}x \in \CO, \,\, \mbox{for all}\,\,t \in [0,T], \,\,  \PP-a.s.$$ therefore, we note that $\mathcal I(u)$ coincides with the $L^1$ norm of $u$. Therefore, one can conclude from \eqref{eqn uL1} the following inequality holds
\begin{align}\label{eqn uL1.1}
\EE\Big[ \sup_{0 \leq s \leq t} |u(s)|^p_{L^1}\Big]
&\leq \EE[|u_0|^p_{L^1}]e^{(\gamma \,p+ \frac{ C_p \gamma_1^{\frac{p}{2}} }{2})T^{p}}.
\end{align}
%$\mathcal I(u(t))=\mathcal I_1(u(t))\,\,  \PP-a.s$. 
\item \newcommand{\trace}{aa}
We will apply the It\^o formula to the functions
\begin{align}\label{def phi psi}
\phi(w)=\tfrac{1}{2}|w|_{L^2}^2, \,\, w\in L^2(\CO) \quad  \hbox{and} \quad \psi(w)=\tfrac{1}{2}|\nabla w|_{L^2}^2, \,\,w\in H^1_4(\CO)
\end{align}
to the processes $u(s)$ and $v(s)$ for $s \in [0, T]$ respectively.
 %, if we take into account the representation of the Wiener process given in Assumption \ref{wiener}.
A straight application of the It\^o formula, however, would require a regularization of the unbounded operators appearing in equation \eqref{chemonoise}.
Nevertheless, one can overcome this problem by the following standard procedure:
We replace the unbounded operator $A:=\Delta$ in equation \eqref{chemonoise} is by its Yosida approximation $A_{\eps}:=\frac 1 \ep (\id-(\eps \Delta- \id)^{-1})$ and the divergence operator by intertwining a mollifier $\phi_\ep$. Thereupon, we let the limit $\ep \to 0^{+}$ (see e.g.\ proof of the  Lemma 3.2 in \cite{marinelli}). Here, for any $\nu \geq 2$ and $\delta\in\RR$, we infer that
\[
 |\DeltaA-\DeltaA_\ep|_{\mathcal{L}(H^{\delta+\nu},H^{\delta})}\to 0 \quad \hbox{as} \quad \ep\to 0. \] To approximate the divergence operator, we now choose a compactly supported function $\varphi\in C^{(2)}_b(\RR^2)$  such that
  $$ \int_{\RR^2}\varphi (x)\mathrm {d} x=1, \quad
  \hbox{and}\quad \lim _{\epsilon \to 0}\varphi _{\epsilon }(x)=\lim _{\epsilon \to 0}\epsilon ^{-2}
\varphi (x/\epsilon )=\delta_0 \quad \hbox{(in distribution)}.$$
Let $\mbox{div}_\ep(w):= \mbox{div}(\varphi_\ep\star w)$ where $\star$ denotes the convolution.
We consider the following system of equations
\[
\begin{cases}
\hspace{1.4cm}
du_{\eps} (t)&=\Big( r_u\Delta u_{\ep} (t)-\chi \di_{\eps}(u_{\eps}(t) \nabla v_{\eps}(t) )\Big)dt
 +  u_{\eps} (t)\, d \mathcal W_1(t)+\gamma u_{\eps}(t) dt
\\
\hspace{1.4cm}
d v_{\eps} (t)&=\Big( r_v \Delta v_{\eps} (t)+\beta u_{\eps} (t)
-\alpha v_{\ep} (t)\Big)dt +  v_{\eps} (t)\, d \mathcal W_2(t)
\\
\hspace{0.4cm} ( u_{\eps}(0),  v_{\ep}(0)) &= (u_0,v_0)
.
\end{cases}
\]

%\end{subequations}
Before applying the It\^o-formula we see to it that the integrands belong to the appropriate space (see  \cite[Chapter 4.4]{pratozab}). Thereupon, one can directly apply the It\^o formula to the functions $\phi$ and $\psi$ defined in \eqref{def phi psi} for the processes $(u_{\eps}(t),v_{\eps}(t))$ for $t \in [0,T]$.
In addition, by the definition of $\gamma_1$  %:= \sum_{k \in \NN} \lambda_k^2 |\psi^1_k|_{L^\infty}^2$,
 we know
$$\mbox{Tr}\big[\phi_{xx}(u)(uQ^\frac 12)(uQ^\frac 12)^T\big]\le \gamma_1^2 |u|_{L^2}^2.
$$
%for all $t\in[0,\tau_{n,n_1,n_2})$.
%For simplicity we denote in the next paragraph $u$ and $\bvnn$ by $u$ and $v$.
%\renewcommand{u}{u}
%\renewcommand{v}{v}
Again, let us remind the abbreviation (see \eqref{defofgamma})
$\gamma_j:= \sum_{k \in \NN} |\lambda_k^j|^2 |\psi_k|_{L^\infty}^2;\,\,j=1,2$.
Let us begin with the following
\begin{align}\label{upart}
& \frac{1}{2} |u (t)|_{L^2}^2 - \frac{1}{2} |u_0|_{L^2}^2+r_u \int_0^t |\nabla u (s)|_{L^2}^2 ds
\nonumber
\\
&=\chi \int_0^ t \langle \nabla u (r), u (r) \nabla v (r) \rangle dr
+ \int_0^ t  \langle u (r) , u (r)\,d \mathcal W_1(r)\rangle
\nonumber\\&{}
+\frac{1}{2}\int_0^t \mbox{Tr}\big[\phi_{xx}(u(r))(u(r)Q^\frac 12)(u(r)Q^\frac 12)^T\big]\, ds
+\int_0^t \la u(r),\gamma u(r)\ra \, dr
 \notag\\
& \leq \ep \int_0^ t |\nabla u (r)|_{L^2}^2 dr
+\frac{\chi^2}{4\ep } \int_0^ t  |u (s)\nabla v (s)|_{L^2}^2 ds
 +(\gamma+ \frac{ \gamma_1^2}{2}) \int_0^t |u (s)|_{L^2}^2 ds
\nonumber\\
&\qquad {}+ \int_0^t  \sum_{k=1}^\infty \lambda^{(1)}_k \langle  u (s),u (s)\psi_k \rangle d\beta_{k}^{(1)}(s),
\end{align}
for $t \in[0,T]$. The Burkholder-Davis-Gundy inequality gives 
\begin{align*}
   & \EE \Big[\int_0^t  \sum_{k=1}^\infty \lambda^{(1)}_k \langle  u (s),u (s)\psi_k \rangle d\beta_{k}^{(1)}(s)\Big]
\le\EE \lk[ \int_0^t  \sum_{k=1}^\infty (\lambda _k^{(1)})^2 |\langle  u (s),u (s)\psi_k \rangle |^2 ds\rk]^\frac 12 
\notag
\\
&
\le \gamma_1\, \EE \lk[ \int_0^t  |u (s)|_{L^2}^4 ds\rk]^\frac 12
\le \gamma_1\, \EE \lk[\lk( \sup_{0\le s\le t}   |u (s)|^2_{L^2} \rk)^\frac 12 \lk( \int_0^t  |u (s)|_{L^2}^2 ds\rk)^\frac 12\rk].
\end{align*}
The Cauchy-Schwarz and the Young inequality gives that for any $\ep_1>0$ there exists a constant $C(\ep_1)>0$ such that
\begin{align}\label{unoise}
   & \EE\Big[\int_0^t  \sum_{k=1}^\infty \lambda_k^{(1)} \langle  u (s),u (s)\psi_k \rangle d\beta_{k}^{(1)}(s)
   \Big]
\le\ep_1 \EE \Big[ \sup_{0\le s\le t}   |u (s)|^2_{L^2}\Big] +C(\ep_1)\EE \Big[ \int_0^t  |u (s)|_{L^2}^2\, ds \Big].
\end{align}
Similarly, by applying the It\^o-formula to $[0,T]\ni t\mapsto \la \nabla v(t),\nabla v(t)\ra$  we obtain % for the process $v(t)$ for $t \in[0,T]$
\begin{align} \label{ito v l4 old}
& \frac{1}{2} |\nabla v (t)|_{L^2}^2-\frac{1}{2} |\nabla v_0|_{L^2}^2
 \leq -r_v \int_0^t |\Delta v(s)|_{L^2}^2 ds
- \alpha \int_0^t |\nabla v(s)|_{L^2}^2 \,ds
\\\nonumber
& \quad
+ \beta \int_0^t \langle \nabla u(s) , \nabla  v(s) \rangle \,ds
+\int_0^t  \sum_{k=1}^\infty \lambda_k^{(2)} \Big \langle  \nabla v (s),\nabla (v (s)\psi_k) \Big\rangle d\beta_{k}^{(2)}(s)
\\\nonumber
& \quad
 +\frac{\gamma_2^2}{2}\int_0^t |\nabla v(s)|_{L^2}^2 \, ds.
\end{align}
Due to the Young inequality, for any $\ep_2>0$ we have 
 \begin{align} \label{ito v l4 old}
& \frac{1}{2} |\nabla v (t)|_{L^2}^2
+r_v \int_0^t |\Delta v(s)|_{L^2}^2 ds
+ \alpha \int_0^t |\nabla v(s)|_{L^2}^2 \,ds
\\\nonumber
& \quad
 \leq \frac{1}{2} |\nabla v_0|_{L^2}^2
 +\frac {\beta^2}{2\ep_2} \int_0^t |\nabla v(s)|_{L^2}^2 ds
+ \ep_2  \int_0^t |\nabla u(s)|_{L^2}^2 \, ds
 +\frac{\gamma_2^2}{2}\int_0^t |\nabla v(s)|_{L^2}^2 \, ds
 \\\nonumber
& \quad
{} + \int_0^t  \sum_{k=1}^\infty \lambda_k^{(2)} \Big\langle  \nabla v (s),\nabla (v (s)\psi_k) \Big\rangle d\beta_{k}^{(2)}(s)
. %,\int_0^t \int_{\CO}v(s,x) \Delta v(s,x) \,d\mathcal{W}_2(s,x).
\end{align}
%Integration by parts and taking into account the Neumann boundary conditions,
%gives $$ \langle  \nabla v (s),\nabla (v (s)\psi_k) \rangle=- \langle  \Delta  v (s),v (s)\psi_k \rangle.$$
%
Using the Burkholder-Davis-Gundy inequality, we obtain
\begin{align*}
%\label{dd}
   & \EE\Big[\sup_{0 \leq s \leq t} \Big|\int_0^t  \sum_{k=1}^\infty \lambda_k^{(2)} \langle  \nabla v (s),\nabla (v (s)\psi_k) \rangle d\beta_{k}^{(2)}(s) \Big|    \Big]
\leq C\, \EE\Big[\int_0^t  \sum_{k=1}^{\infty} (\lambda^{(2)}_k)^2
\Big|  \langle  \nabla v (s),\nabla (v (s)\psi_k) \rangle \Big|^2 \, ds \Big]^{\frac{1}{2}}.
\end{align*}
The H\"older inequality gives
\begin{align*}
&\EE\lk[\sup_{0 \leq s \leq t} \lk|\int_0^t  \sum_{k=1}^\infty \lambda_k^{(2)} \langle  \nabla v (s),\nabla (v (s)\psi_k) \rangle d\beta_{k}^{(2)}(s) \rk|    \rk]
\\
& \leq C \, 
\EE\lk[\int_0^t  \lk(\sum_{k=1}^{\infty} \lambda^{(2)}_k |\psi_k|_{L^\infty}^2\rk)
 |\nabla v(s)|^4_{L^2}  \, ds \rk]^{\frac{1}{2}}.
%& \leq C\, \EE\Big[ \int_0^t \Big( \int_{\CO}   |\nabla v (s,x)|^2 \, dx \Big)^2 \, ds  \Big]^{\frac{1}{2}}
% \leq \frac{1}{4} \EE\Big[\sup_{0 \leq s \leq t} |\nabla v(s)|_{L^1}^2 \Big]
%+ C\, \EE\Big[\int_0^t |\nabla v(s)|_{L^1}^2 \Big].
\end{align*}
Applying the H\"older inequality, and then the Young inequality gives that for all $\ep_3>0$ there exists a constant $C(\ep_3)>0$ such that
\begin{align}\label{wienernoisepart}
&\EE\lk[\sup_{0 \leq s \leq t} \lk|\int_0^t  \sum_{k=1}^\infty \lambda_k^{(2)} \langle  \nabla v (s),\nabla (v (s)\psi_k) \rangle d\beta_{k}^{(2)}(s) \rk|    \rk]
\notag
\\
& \leq C \,
\EE\lk[\sup_{0\le s\le t} |\nabla v(s)|^2_{L^2} \, \int_0^t  \lk(\sum_{k=1}^{\infty} \lambda^{(2)}_k |\psi_k|_{L^\infty}^2\rk)
 |\nabla v(s)|^2_{L^2}  \, ds \rk]^{\frac{1}{2}}
\notag
\\
& \leq C \,
\EE\lk[\lk( \sup_{0\le s\le t} |\nabla v(s)|^2_{L^2} \rk)^\frac 12 \,\lk( \int_0^t 
 |\nabla v(s)|^2_{L^2}  \, ds \rk)^{\frac{1}{2}}\rk]
\notag
\\
& \leq C \,{\gamma_2} \lk[\EE\sup_{0\le s\le t} |\nabla v(s)|^2_{L^2} \rk]^\frac 12 \,\lk[ \EE \int_0^t
 |\nabla v(s)|^2_{L^2}  \, ds \rk]^{\frac{1}{2}}
\notag
\\
& \leq \ep_3\, \EE\Big[\sup_{0\le s\le t} |\nabla v(s)|^2_{L^2}\Big] +C(\ep_3)\,C \,\gamma_2\, \EE\Big[ \int_0^t
 |\nabla v(s)|^2_{L^2}  \, ds\Big].
\end{align}  %(see \eqref{u nab v})
%\begin{align}\label{u nab v}
%| u (s) \nabla v (s)|^2_{L^2} &= | u (s)|_{L^4}^2 | \nabla v (s)|^2_{L^4}
%\leq 
% C |u(s)|_{H^1_2} |u(s)|_{L^1} |\nabla v(s)|_{H^1_2}^{\frac{1}{2}} |\nabla v(s)|_{L^2}^{\frac{3}{2}}
% \notag
%\\
%&\le C |u(s)|_{H^1_2} |u(s)|_{L^1} |v(s)|_{H^2_2}^{\frac{1}{2}} |v(s)|_{H^1_2}^{\frac{3}{2}}.
%\end{align}
%%
%%\leq C |u(s)|_{H^1_2} |u(s)|_{L^1} |v(s)|_{H^2_2}^{\frac{1}{2}} |v(s)|_{H^1_2}^{\frac{3}{2}}
%The Young inequality gives that for any $\ep_3>0$ there exist constants $C(\ep_3)>0$ such that
%\begin{align}\label{vuterm}
%& \leq \eps_3 \Big(  |u(s)|_{H^1_2}^2 + |v(s)|^2_{H^2_2}  \Big) + C({\eps_3}) |u(s)|_{L^1}^4 |v(s)|_{H^1_2}^6
%\notag\\
%& \leq \eps_3 \Big(  |u(s)|_{H^1_2}^2 + |v(s)|^2_{H^2_2}  \Big) + C({\eps_3}) |u(s)|_{L^1}^8
%+ C({\eps_3}) |v(s)|_{H^1_2}^{12}
%\end{align}
%
%
%
Now, let us consider the term $ | u (s) \nabla v (s)|^2_{L^2}$ in \eqref{ito v l4 old}.
Using the following interpolation inequality (see \cite[p.\ 233]{Brezis})
\[
|\cdot|_{L^q} \leq C_{p,q} |\cdot|_{H^1_2}^\theta |\cdot|_{L^p}^{1 - \theta}  \quad \hbox{for} \quad 1 \leq p< q \leq \infty, \,\, \theta=\frac{\frac{1}{p} - \frac{1}{q}}{\frac{1}{p} + \frac{1}{2}}=\frac 12,
\] 
for once $q=4,\,\,p=1$ and once  $q=4,\,\,p=2$, we achieve
%Using H\"older inequality with $\frac{2}{3}+ \frac{1}{3}=1$ and Young's inequality for $\ep_1>0$, we obtain
\begin{align}\label{u nab v}
| u (s) \nabla v (s)|^2_{L^2} &= | u (s)|_{L^4}^2 | \nabla v (s)|^2_{L^4}
\leq 
 C |u(s)|_{H^1_2} |u(s)|_{L^1} |\nabla v(s)|_{H^1_2}^{\frac{1}{2}} |\nabla v(s)|_{L^2}^{\frac{3}{2}}
 \notag
\\
&\le C |u(s)|_{H^1_2} |u(s)|_{L^1} |v(s)|_{H^2_2}^{\frac{1}{2}} |v(s)|_{H^1_2}^{\frac{3}{2}}.
\end{align}
%
%\leq C |u(s)|_{H^1_2} |u(s)|_{L^1} |v(s)|_{H^2_2}^{\frac{1}{2}} |v(s)|_{H^1_2}^{\frac{3}{2}}
The Young inequality gives that for any $\ep_4>0$ there exists constant $C(\ep_4)>0$ such that
\begin{align}\label{vuterm}
& \leq \eps_4 \Big(  |u(s)|_{H^1_2}^2 + |v(s)|^2_{H^2_2}  \Big) + C({\eps_4}) |u(s)|_{L^1}^4 |v(s)|_{H^1_2}^6
\notag\\
& \leq \eps_4 \Big(  |u(s)|_{H^1_2}^2 + |v(s)|^2_{H^2_2}  \Big) + C({\eps_4}) |u(s)|_{L^1}^8
+ C({\eps_4}) |v(s)|_{H^1_2}^{12}
\end{align}
The main difficulty is to handle the non-linear term $u \nabla v$.
 To estimate the non-linear term which is $\EE\Big[\sup_{0 \leq s \leq T}|u(s) \nabla v(s) |^2_{L^2}\Big]$, we need to obtain bound for $\EE\Big[ \sup_{0 \leq s \leq t} |u(s)|^8_{L^1}\Big]$ and $\EE\Big[ \int_0^t |v(s)|^{12}_{H^1_2}\,ds\Big]$. 
% {\bf I do not understand what you want here!}
%$v$ satisfies
%\[d{v}(t) = \big(r_v \Delta v(t)+\beta u(t) -\alpha v(t)\Big)\, dt + v(t) d \mathcal W_2(t); \quad
%v(0)=v_0.
%\]
We now apply Theorem 4.5 in \cite{vanNeerven1}, for $p>1$
with $X_0=(H^{1}_2(\CO))^{\ast}$ (the dual space of ${H}_{2}^{1}(\CO)$
corresponding to Neumann boundary conditions), $X_1=H^{1}_2(\CO)$,
and $(X_0,X_1)_{1-\frac{1}{p},p}=B_{2,p}^{1-\frac{2}{p}}$. Hence, using the embedding $L^1(\CO) \hookrightarrow (H^{1}_2(\CO))^{\ast}$ in one dimension, and using \eqref{eqn uL1} we obtain
\begin{align}\label{esti v van}
 \EE\Big[\int_0^t |v(s)|_{H^1_2}^p \,ds   \Big]
 &\leq C\Big(\EE|v_0|^p _{B_{2,p}^{1-\frac{2}{p}}}
+ \EE\Big[\int_0^T |u(s)|_{(H^1_2)^\ast}^p\,ds    \Big]      \Big)
\notag\\
&\leq  C\Big(\EE|v_0|^p _{B_{2,p}^{1-\frac{2}{p}}}
+ \EE\Big[\int_0^T |u(s)|_{L^1}^p\,ds    \Big]      \Big)
\notag\\
&
\leq  C\Big(\EE|v_0|^p_{B_{2,p}^{1-\frac{2}{p}}}
+T\, \EE|u_0|_{L^1}^p e^{(\gamma p + \frac{1}{2})T^p}\Big).
\end{align}
We set $\ep_1<\frac{1}{4}$, $\eps_2<\frac{r_u}{8}$, $\eps_3< \frac{1}{4}$,
$\eps_4<\{\frac{r_u}{8} \wedge \frac{r_v}{4}\}$.
Taking supremum and then expectation in \eqref{upart}, \eqref{ito v l4 old},  and \eqref{vuterm} and combining \eqref{unoise}, \eqref{wienernoisepart}, and \eqref{esti v van} we finally arrive at the following estimate
\begin{align}
& \frac 14 \,\EE \Big[\sup_{0 \leq s \leq t  }|u(s)|^2_{L^{2}} \Big]
+ \frac 14 \,\EE \Big[\sup_{0 \leq s \leq t }|\nabla v(s)|^2_{L^{2}} \Big]
 + \frac{r_u}{4}\,\EE\Big[ \int_0^{t }|\nabla u(s)|^2_{L^{2}}\, ds \Big]
\notag\\
& \quad
 + \frac{r_v}{2}\, \EE\Big[ \int_0^ {t }\int_{\CO} |\Delta v(s,x)|^2\,dx\, ds\Big]
  + \alpha \EE\Big[ \int_0^ {t }\int_{\CO} |\nabla v(s,x)|^2\,dx\, ds\Big]
 % + r_v\, \EE\Big[ \int_0^ {t \wedge \tau_{n}(u)}\int_{\CO} |\nabla v(s,x)|^2 |D^2v(s,x)|^2\,dx\, ds\Big]
 \notag\\
&
 \leq C \Big(\frac 12 \EE|u_0|_{L^2}^2
+\frac{1}{4} \EE|\nabla v_0|_{L^2}^2
 + T\, \EE|u_0|_{L^1}^8 \, e^{(\gamma^8 + \frac{1}{2}) T^8}
 + T \EE|v_0|^{12}_{B_{2,12}^{\frac{5}{6}}}
 \notag\\
& \quad + T^2 \EE |v_0|_{L^1}^{12} \, e^{(\gamma^{12} + \frac{1}{2}) T^{12}}\Big)\, e^{(\gamma +\frac{\gamma_1^2}{2}
+ \frac{\gamma_2^2}{2}  + \frac{\beta}{r_v}) T}.
\end{align}
 %end{align}
 %
%\red{Again, this uses (time-uniform) bounds on $|\nabla v|_{L2}^2.$}
\del{Let us remind, that in  \cite{EH+DM+TT_2019} we have showed an upper bound of the functional
\DEQS %Z\label{wdef}
\CF(u,v)=
W (u,v)= \int_\CO \Big(u(x) \log u(x) - \rho u(x)v(x)\Big) dx +
C_1|\nabla v|^2_{L^2} + C_2|v|_{L^2}^2,\quad u\in \LLa(\CO),\,\, v\in H^1(\CO).
\EEQS}

\end{steps}

This completes the proof.
\end{proof}

%\section{Maximal inequalities for the $v$ term to be summerized}
\appendix
\section{Technical lemmata}\label{TL}

%\begin{lemma}\label{lem:porous-medium-inequality}
%
%Let $q>1$, $x,y\in\RR$. Then the following inequality holds
%\begin{equation}\label{eq:porous-medium-inequality}
%(x^{[q]}-y^{[q]})(x-y)\ge2^{1-q}|x-y|^{q+1},\end{equation}
%where $z^{[q]}:=|z|^{q-1}z$ for $z\in\RR$.
%\end{lemma}

%\begin{proof}
%See e.g. \cite[Lemma 3.1]{Liu:2009ih}.
%\end{proof}
\begin{lemma}\label{lem:gajewski}
$\phi(x)=x(\ln(x)-1)$, $x\ge 0$, satisfies
\DEQS
\phi(x)-2\phi\lk( \frac {x+y}2\rk) +\phi(y)\ge \frac 14 \lk(\sqrt{x}-\sqrt{y}\rk)^2,\quad x,y\ge 0.
\EEQS
\end{lemma}

\begin{proof}
See e.g. \cite[Lemma 5.1]{Gajewski_1994} for more details.
\end{proof}

To evaluate $S_2$ in the proof of Theorem \ref{thm path uniq}, we formulate the following  Claim.
% We first note that if $u_1=u_2=0$, then
%\DEQS
%S_2= \gamma
%\int_0^t \int_{\CO}
%\Big[
%u_1(s) \ln \lk( \dfrac{2u_1(s)}{u_1(s)+u_2(s)} \rk)
%+ u_2(s) \ln \lk(\dfrac{2u_2(s)}{u_1(s) + u_2(s)} \rk)\Big]\,dx\,ds
%=0.\notag
%\EEQS
%Otherwise, we note that $S_2$ can be written as
%\DEQS
%S_2 &=&  \lk\{ \barray
%\mathlarger{\int_0^t \int_{\CO}}
%\gamma\Big[
%u_1(s)\Big\{ \ln \Big( \dfrac{2}{1 + \frac{u_2(s)}{u_1(s)}} \Big)
%+ \lk( \frac{u_2(s)}{u_1(s)} \rk)
%\ln \Big(
%\dfrac{ 2 \lk(\frac{u_2(s)}{u_1(s)}\rk)}
%{1 + \frac{u_2(s)}{u_1(s)}} \Big)\Big\}\Big]\,dx\,ds, \quad \mbox{if} \quad u_1(s) \neq 0,
%%&\dot{u}(t)  \mbox{div}( D_u(u(t),v(t)) \nabla u(t)+\mbox{div} \chi(u(t),v(t)) u(t)\nabla v(t)+H(u(t),v(t))\, \phantom{\Big|}
%\\
%\mathlarger{ \int_0^t \int_{\CO}}
%\gamma \Big[
%u_2(s)\Big\{ \frac{u_1(s)}{u_2(s)}
%\ln \Big(
%\dfrac{ 2 \lk(\frac{u_1(s)}{u_2(s)}\rk)}
%{1 + \frac{u_1(s)}{u_2(s)}}\Big)
%+
%\ln \Big(\dfrac{2}{1 + \frac{u_1(s)}{u_2(s)}}\Big) \Big\}
% \Big]\,dx\,ds, \quad \mbox{if} \quad u_2(s) \neq 0.
%\phantom{\big|}
%\earray\rk.
%\EEQS
%In other way, we can write

\begin{claim}\label{claim nonneg}
Let $f:[1,\infty) \rightarrow \mathbb{R}$ be a function given by
\begin{align}\label{def f}
f(u)= \ln \lk(\dfrac{2}{1+u} \rk) + u\ln \lk(\dfrac{2u}{1+u} \rk)- (\sqrt{u}-1)^2, \quad u \in [1,\infty).
\end{align}
Then,  $f(u) \leq f(1)=0 \quad \hbox{for} \quad u \in [1,\infty)$.
\end{claim}
\begin{proof}[Proof of Claim \ref{claim nonneg}]
Let us first evaluate the first two derivatives of $f$. On differentiating $f(u)$ and $f'(u)$ with respect to $u$, we get
\begin{align*}
f'(u)= \ln \lk(\dfrac{2u}{1+u} \rk) + \frac{1}{\sqrt{u}} -1, \quad \quad
f''(u)= \frac{2\sqrt{u} - (1+ u)}{2u^{3/2}(1+u)}
= \frac{-(\sqrt{u}-1)^2}{2 u^{3/2}(1+u)} \leq 0, \quad \hbox{for} \quad u \in [1,\infty).
\end{align*}
This implies that $f'$ is a decreasing function on $[1,\infty)$, i.e.\
$f'(u) \leq f'(1)=0 \,\, \hbox{for} \,\, u \in [1,\infty)$. This again yields that
%$f$ is a decreasing function on $[1,\infty)$, i.e.\
$f(u) \leq f(1)=0 \,\, \hbox{for} \,\, u \in [1,\infty)$. This completes the proof of the claim.

\end{proof}

%{\color{red}what is happening if $u_1(x)=0$ on an open subset $\mathcal{D}$.}
%The following inequality holds
%\begin{align}\label{posi esti}
%u_1^q(x,s)\ln \lk( \dfrac{2u_1(x,s)}{u_1(x,s) + u_2(x,s)} \rk) + u_2^q(x,s) \ln \lk(\dfrac{2u_2(x,s)}{u_1(x,s) + u_2(x,s)} \rk)
%\geq 0 \quad x \in \CO,\,\, s \in [0,T].
%\end{align}

\end{document}